\newif\ifdraft
\draftfalse

\documentclass[reqno]{amsart} 

\input figdir.local 

\usepackage[dvipsnames]{xcolor}
\usepackage{graphicx}
\usepackage {amsmath}
\usepackage {amssymb}
\usepackage {amsthm}
\usepackage[bookmarks=true]{hyperref}

\newcommand{\activereg}{A}

\newcommand{\affinetransformation}{\afftransf}
\newcommand{\afftransf}{{\mathcal R}}
\DeclareMathOperator{\BV}{BV}

\newcommand{\eps}{\varepsilon}

\newcommand{\firstquadrant}{\overline {\rm I}_{{\rm quad}}}
\newcommand{\firstquadrantopen}{{\rm I}_{{\rm quad}}}

\newcommand{\freebound}{\Gamma}

\newcommand{\halfline}{\mathfrak{h}}
\newcommand{\incrt}{s}
\newcommand{\inidat}{{u_0}}

\newcommand{\inverseMau} {S}

\newcommand{\Lyap}{\mathfrak L}

\newcommand{\mass}{m}

\newcommand{\nada}[1]{}
\newcommand{\NN}{\mathbb N}
\newcommand{\oi}{[0,+\infty)}
\newcommand{\oic}{[0,+\infty)}

\newcommand{\pressol}{v}
\newcommand{\primoinsieme}{E}
\newcommand{\R}{\ensuremath{\mathbb{R}}}

\newcommand{\secondoctantopen}{\soctopen}
\newcommand{\secondoinsieme}{F}
\newcommand{\soct}{\overline{{\rm II}}_{{\rm oct}}}
\newcommand{\soctopen}{ {\rm II}_{\rm oct} }

\newcommand{\supp}{\textup{supp}}
\newcommand{\support}{c}
\newcommand{\tailreg}{T}

\newcommand{\uM} {\unknownMau}
\newcommand{\unknownMau} {\ell}
\newcommand{\unknownplustMau} {L}

\newcommand{\wu} {\widehat u}
\newcommand{\lmin}{m}
\newcommand{\lmax}{M}
\newcommand{\tmin}{t}
\newcommand{\tmax}{T}
\newcommand{\gmin}{g}
\newcommand{\gmax}{G}
\newcommand{\lmins}{\lmin^*}
\newcommand{\lmaxs}{\lmax^*}
\newcommand{\tmins}{\tmin^*}
\newcommand{\tmaxs}{\tmax^*}
\newcommand{\gmins}{\gmin^*}
\newcommand{\gmaxs}{\gmax^*}

\def\len{\mathop{\rm len}}
\def\M{\mathcal{M}}

\def\N{\mathbb N}

\theoremstyle{plain}
\numberwithin{equation}{section}
\newtheorem{Remark}{Remark}[section]
\newtheorem{Theorem}{Theorem}[section]
\newtheorem{Proposition}{Proposition}[section]
\newtheorem{Definition}{Definition}[section]

\newtheorem{Lemma}{Lemma}[section]

\theoremstyle{definition}
\newtheorem{Example}{Example}[section]

\def\<#1>{\lower4pt\hbox{$\,\includegraphics{\figdir/mancala-#1.mps}\,$}}
\def\(#1){\includegraphics{\figdir/mancala-#1.mps}}
\def\<<#1>>{\vcenter{\hbox{\includegraphics{\figdir/riemann-#1.mps}}}}

\newcommand{\alessandro}[1]{\ifdraft{\color{orange}#1}\fi}
\newcommand{\giovanni}[1]{\ifdraft{\color{blue}#1}\fi}
\newcommand{\maurizio}[1]{\ifdraft{\color{ForestGreen}#1}\fi}
\newcommand{\general}[1]{\ifdraft #1\fi}

\author[G.\ Bellettini]{Giovanni Bellettini}
\address{Dipartimento di Ingegneria dell'Informazione e Scienze Matematiche,
Universit\`a di Siena, 53100 Siena, Italy, and International Centre for
Theoretical Physics ICTP, Mathematics Section, 34151 Trieste, Italy}
\email{bellettini@diism.unisi.it}

\author[A.\ Betti]{Alessandro Betti}
\address{Universit\'e 
C\^ote 
d'Azur, Inria, CNRS, Laboratoire I3S,
Maasai team, Nice, France}
\email{alessandro.betti@inria.fr}

\author[M.\ Paolini]{Maurizio Paolini}
\address{Dipartimento di Matematica e Fisica, Universit\`a Cattolica del
Sacro Cuore, 25121 Brescia, Italy}
\email{maurizio.paolini@unicatt.it}

\date{\today\general{\qquad
\giovanni{{\vrule height7pt width10pt depth3pt}} $\mapsto$ Giovanni\qquad
\maurizio{{\vrule height7pt width10pt depth3pt}} $\mapsto$ Maurizio\qquad
\alessandro{{\vrule height7pt width10pt depth3pt}} $\mapsto$ Alessandro}}
\thanks{The first author is member of the GNAMPA of INdAM.
This work has been supported by the French government, 
through the 3IA C\^ote  d'Azur, Investment in the Future, project managed by the National Research Agency (ANR) 
with the reference number ANR-19-P3IA-0002.}

\title[A free boundary singular transport equation]{A free boundary singular transport equation as a formal limit of a 
discrete dynamical system}

\begin{document}
\sloppy 

\maketitle

\begin{abstract}
We study the continuous version of a hyperbolic rescaling
of a discrete game, called open mancala. The resulting PDE
turns out to be a singular transport equation, with a 
forcing term taking values in $\{0,1\}$, and 
discontinuous in the solution itself. We prove
existence and uniqueness of a certain formulation of the problem,
based on a nonlocal equation satisfied by the free boundary
dividing the region where the forcing is one (active region)
and the region where there is no forcing (tail region). Several
examples, most notably the Riemann problem,
are provided, related to singularity formation. Interestingly,
the solution 
can be obtained by a suitable vertical rearrangement of a multi-function.
Furthermore, the PDE admits a Lyapunov functional.
\end{abstract}

\section{Introduction}\label{sec:introduction}
The mancala game is a game that, in its open and idealized version, can be
described in terms of two moves as follows. Suppose to have an infinite
sequence of holes on the half-line $(0,+\infty)$, each hole 
containing zero or more
seeds (always in finite number). Suppose also that the first hole is
nonempty, that nonempty holes are consecutive and in finite number.  The
first half-mancala move consists in taking all seeds in the first hole (the
left-most nonempty hole, numbered $1$), and sowe them in the subsequent
holes, one seed per hole.  In this way the first hole becomes empty; the
second half-mancala move consists in a left-translation of one position of
all subsequent holes, so that the new left-most hole becomes nonempty, and
the two half-moves can therefore be repeated. Of course, the total mass
(i.e., the total number of seeds) is preserved.  The open mancala game is
modeled as a discrete dynamical system, for which several interesting
questions can be posed, such as the classification of periodic
configurations, the optimal number of moves necessary to reach a periodic
configuration, and so on, see for instance \cite{Pao} and references therein.
The aim of this paper is to analyse a continuous version of this dynamical
system, which turns out to be modeled by a new type of transport
equation.

After having introduced a natural discrete time, a suitable hyperbolic
rescaling of the discrete game and a simple algebraic manipulation (see
Section \ref{sec:the_discrete_game} for the details) lead to
\begin{equation}
\label{eq:discrete_equation_h_intro}
\frac{u_{ih}^{(k+1)h}-u_{ih}^{kh}}{h}=
\frac{u_{ih}^{(k+1)h}-
u_{(i-1)h}^{(k+1)h}}{h}+
1_{\left\{2h\leq\,\cdot\,\leq u_{h}^{kh} + h \right\}}(ih),
\end{equation}
where $ih$, $i \geq 1$ integer, denotes the discrete $i$-th space position,
$kh$, $k \geq 0$ integer, denotes the discrete $k$-th time, $h$ is the space
grid (equal to the time grid), $u_{jh}^{lh}$ stands for the number of seeds
at $(jh, lh)$, and $1_{A}$ is the characteristic function of the set
$A\subset [0,+\infty)$. The discrete system
\eqref{eq:discrete_equation_h_intro} is coupled with
\begin{equation}\label{eq:initial_condition_intro}
u_{ih}^0 = \inidat(ih), \qquad i \geq 1,
\end{equation}
$u_0$ being the given initial configuration.

It is then natural to consider the following hyperbolic pde, considered as a
formal limit of \eqref{eq:discrete_equation_h_intro},
\eqref{eq:initial_condition_intro} as $h \to 0^+$:
\begin{equation}\label{eq:pde_intro}
\begin{cases}
u_t(t,x)=u_x(t,x)+1^{}_{(0,u(t,0))}(x), & {} \\
u(0) = u_0. & {}
\end{cases}
\end{equation}
Notice carefully the $\{0,1\}$-valued discontinuous forcing term on the
right-hand side which, given any $t \geq 0$, depends on the space right trace
$u(t,0)$ of $u(t,\cdot)$ at $x=0$; this space dependent forcing depends on
$u$ in a nonlinear and nonlocal way, and results the PDE in
\eqref{eq:pde_intro}, to our best knowledge, into a new type of partial
differential equation.  Observe also that, {\it formally}, the total mass is
conserved:
\[\begin{aligned}
& \frac{d}{dt}
\int_0^{+\infty} u(t,x)~ dx =
\int_0^{+\infty} u_t(t,x)~   dx\\
= &
\int_0^{+\infty} \left(u_x(t,x)  +
1_{(0,u(t,0))}(x)\right)~ dx =
- u(t,0) +
\int_0^{+\infty} 1_{(0,u(t,0))} (x)~dx = 0.
\end{aligned}\]
The physical interpretation of \eqref{eq:pde_intro} can be given in terms of
mass transportation, as follows.  Imagine to have an horizontal conveyor belt
transporting, at uniform speed, a certain amount of sand toward the left;
once a slice of sand of height $\kappa$ reaches the (left) boundary $x=0$ and
falls down, it is uniformly redistributed horizontally for a length $\kappa$
on the conveyor, and the process continues.

From the mathematical point of view, we are facing a homogeneous linear
transport equation in a ``tail'' region
\[\tailreg(u) := \{\,(t,x) : t\geq 0,\; x > u(t,0)\,\},\]
where the forcing
term is suppressed,
and an inhomogeneous 
transport equation in an ``active'' region
\[\activereg(u) := \{\,(t,x) : t\geq 0,\; 0<x < u(t,0)\,\},\]
where the forcing term equals one.  Tail and active regions are separated by
an interface, a curve in time-space in our setting, which is the
(generalized) graph of $u(\cdot,0)$ over $\R^+$, and that can be considered
as a sort of free boundary.

We anticipate here that, interestingly, some similarities with the typical
phenomena of entropy solutions of nonlinear first order conservation laws
(say the Burger's equation, just to fix ideas) appear also for solutions of
\eqref{eq:pde_intro}, such as possible creation of decreasing jumps, and the
validity of a unilateral Lipschitz condition \cite{Daf:16}. Typically,
for Burgers'
equation, increasing jumps disappear instantaneously, while in the
present model they are progressively eroded while travelling toward the
origin\footnote{After reaching the origin, they possibly disappear.} at unit
speed (see Figs. \ref{fig:Riemann_movie_I} and \ref{fig:Riemann_movie_II}).
The most surprising phenomenon happening in the free boundary of
\eqref{eq:pde_intro} is what we have described through an {\it affine
transformation} (equation \eqref{eq:R}) followed by a {\it vertical
rearrangement} (equation \eqref{eq:trasformato_riarrangiato}), a procedure
that reminds the ``equal area rule'' of conservation laws (see \cite[pag. 42]{Wh},
and 
\cite{Br:84}) and that will be carefully analysed in Sections
\ref{subsec:transformation_R_and_rearrangement} and
\ref{sec:the_Riemann_problem}. Roughly, it turns out that the free
boundary is the (generalized) graph $G_\uM$ of a $\BV_{{\rm loc}}$ 
function $\uM$ which has to
satisfy a nonlocal equation (see \eqref{eq:elle_intro}) that can be described
in two steps: first the affine transformation forces the creation of a
multi-valued graph, in correspondence of certain previous parts of the graph
of $\uM$ having slope less than $-1$ (called critical slopes).  In a second
step, \eqref{eq:elle_intro} reduces the multigraph to the graph of a
single-valued function by a geometric principle based on a vertical
rearrangement (a Steiner symmetrization). In turn, from such a genuine
function, the solution $u$ can be directly recovered (Remark
\ref{rem:alternative_expression}). The remarkable operation of rearrangement
is always necessary at a critical time (Definition
\ref{def:critical_segments_and_critical_times}), i.e., a time corresponding
to the presence of a critical slope, and typically such times are
unavoidable.

In order to define and analyse a solution to \eqref{eq:pde_intro}, it is
convenient to change variables in the obvious way,
\begin{equation}\label{eq:linear_change_of_variables}
\Phi\colon(t,x) \in \firstquadrant
\to (\tau,\xi) \in \soct, \qquad
\tau := t, \quad \xi := t+x, \end{equation}
where $\firstquadrant := [0,+\infty) \times \oic$ is the closure of the open
first quadrant $\firstquadrantopen$, $\soct := \{(\tau,\xi) \in [0,+\infty)
\times \R, \xi \geq \tau\}$ is the closure of the open second octant ${\rm
II}_{{\rm oct}}$, and formally consider \eqref{eq:pde_intro} in the new
coordinates, which after integration reads as
\begin{equation}\label{eq:pde_integral_intro}
\wu(\tau,\xi) = \inidat(\xi) + \int_0^\tau 1_{(\xi,+\infty)}
\left(\wu(s,s)+ s\right)\,ds, \qquad (\tau, \xi) \in \soctopen,
\end{equation}
where $\wu(s,s)$ is the right limit of $\wu(s,\cdot)$ at any $s >0$.

Switching our perspective from \eqref{eq:pde_intro} to
\eqref{eq:pde_integral_intro}, we are interested in the global existence,
uniqueness, and qualitative properties of a solution $\wu$: the leading idea
is to try to identify a plane curve\footnote{Heuristically, the graph of $u$
as a function of time, on the line $\{x=0\}$.} potentially representing the
free boundary, and then reconstruct from it the values of $\wu$ in
$\soctopen$.  It turns out that this strategy is successfull (see Sections
\ref{sec:the_continuous_problem:t_x_formulation},
\ref{sec:another_formulation} for detailed statements), the free boundary
being identified by the graph of the function $\ell$ in \eqref{eq:elle_intro}
below\footnote{ How to guess the equation for $\ell$ is obtained in formula
\eqref{eq:mistery}.}; proving existence of a solution to
\eqref{eq:elle_intro} is the central part of the proof.

Denote by $\vert \cdot\vert$ the one-dimensional Lebesgue measure in $\R$.

\begin{Theorem}\label{teo:main_uno}
Let $\inidat \in \BV_{{\rm loc}}(\oi)$
be nonnegative. The nonlocal equation
\begin{equation}\label{eq:elle_intro}
\unknownMau(t) = \inidat(t) +
\vert\{\, \incrt \in [0,t] : \unknownMau(t-\incrt) 
> \incrt\, \}\vert \qquad \forall t \in \oi,
\end{equation}
in the nonnegative unknown
$\unknownMau \in \BV_{{\rm loc}}(\oi)$,
has a solution, which is unique provided $\inidat$
is positive in a right neighbourhood of the origin.
In addition, $\ell$ cannot have increasing jumps, unless $\inidat$
has, while it may have decreasing jumps. 
Moreover, the function 
\begin{equation}\label{eq:global_intro}
\pressol(\tau,\xi):=\inidat(\xi)+|\{\,s\in[0,\tau] : 
\unknownMau(s) + s>\xi\,\}|\qquad \forall(\tau,\xi)\in \soctopen
\end{equation}
is the global unique solution to \eqref{eq:pde_integral_intro} having the generalized
graph $G_\ell$ of $\ell$ as free boundary, it satisfies the mass
conservation, and $\pressol(\tau,\cdot)$ satisfies a unilateral Lipschitz
condition, provided $\inidat$ does.
\end{Theorem}  

Theorem \ref{teo:main_uno} essentially says that the construction of a
solution of \eqref{eq:pde_integral_intro} is equivalent to solve the nonlocal
equation \eqref{eq:elle_intro} describing the free boundary: this is not
immediate, and requires some recursive argument.  As already said, equation
\eqref{eq:elle_intro} has a geometric meaning (see Remark
\ref{rem:geometric_meaning} and Section
\label{subsec:transformation_R_and_rearrangement}) and has also
various equivalent formulations (see the end of Section
\ref{sec:another_formulation}): one of them, see \eqref{eq:b_B_u_0}, is based
on the generalized inverse of the vertical rearrangement of the function
$L_\uM(t) := \ell(t) + t$, and may be used for the explicit construction of
$\ell$ in specific examples.

The main difficulties in the proof of Theorem \ref{teo:main_uno} are due to
the fact that, at a point $(\overline \tau,\overline \xi)$, a characteristic
line $\{\xi = \overline \xi\}$ passing through $(\overline \tau, \overline
\xi)$, may intersect several (possibly infinitely many) times both the
subgraph and the epigraph of $\uM$, before ending on the vertical axis
$\{\tau =0\}$ where keeping the value of $\inidat$; this happens in
correspondence of critical slopes of $\uM$, corresponding (roughly) to
regions where $L_\uM$ becomes strictly decreasing and $L_\uM$ loses
monotonicity. As already said, in correspondence of these critical values
function $\uM$ and the solution $v$ exhibit interesting features. 

It is worthwhile to remark that equation \eqref{eq:pde_intro} admits a
Lyapunov functional, which is suggested by looking at the discrete Lyapunov
functional known for the open discrete mancala \cite{Pao}, see
\eqref{eq:discrete_Lyapunov_functional} and Section
\ref{sec:a_Lyapunov_functional}.

It is immediate to devise a class of stationary solutions for the flow, i.e.,
$u(x)=(\sqrt{2m}-x) 1^{}_{[0,\sqrt{2m}]}$ (Example
\ref{exa:stationary_solutions}), $m \geq 0$ being the initial mass.  In an
interesting example (the Riemann problem corresponding to the initial datum
$\inidat = 1_{[0,1)}$, discussed in Section \ref{sec:the_Riemann_problem})
it turns out that the Lyapunov functional\footnote{A Lipschitz function in
time.} is constant unless in proximity of a critical time, after which it is
strictly decreasing for a while.  The Riemann problem is a nice example that
illustrates a rather complex solution starting from a very simple initial
condition, and that shades light on the main phenomena behind equation
\eqref{eq:pde_intro}, in particular the vertical rearrangement.  As already
explained, the main point is to construct the function $\uM$, and this can be
done recursively, taking advantage of the fact that $\uM$ must be
polygonal\footnote{
One can see that the flow
starting from a polygonal initial condition (i.e., a piecewise affine
possibly
discontinuous function, see Definition
\ref{def:piecewise_affine}) generates a
piecewise affine solution, a useful fact for the explicit
computation of solutions.}
, and constructing the sequences of slopes connecting pairs of local
minima and local maxima. The explicit construction of $\uM$, though not in
closed form, is given by recurrence, and is rather intricated, due to the
unavoidable presence of (a countable number of) critical slopes, and
therefore of a multivalued\footnote{In this case there cannot be more than
three values.} graph; the most involved case is when two vertical
rearragements need to be performed in adjacent segments (see the first
picture in Fig. \ref{fig:slopes}). It is proven in Section 
\ref{sec:the_Riemann_problem} that the corresponding
stationary solution is reached in an infinite time.

The content of the paper is the following. In Section
\ref{sec:the_discrete_game}, after briefly recalling the moves in the open
mancala game, we describe the rescaling leading to
\eqref{eq:discrete_equation_h_intro}. In Section
\ref{sec:the_continuous_problem:t_x_formulation} we consider the $(t,x)$
formulation of the PDE, and its transformed version in $(\tau,\xi)$
variables, leading to the notion of integral solution (Definition
\ref{def:integral_solution}).  In Section \ref{sec:another_formulation} we
introduce the equation solved by $\uM$, the graph of which will be the free
boundary. Global existence and uniqueness of $\uM$, as well as some
qualitative properties, are proven in Theorem \ref{teo:existence_of_l}.  The
geometric meaning of equation \eqref{eq:elle_intro} in terms of the affine
transformation and the vertical rearrangement is explained in Remark
\ref{rem:geometric_meaning} and Section
\ref{subsec:transformation_R_and_rearrangement}. The last part of Section
\ref{sec:another_formulation} is concerned with a more general class of
initial data, in particular those that may vanish at zero. In Section
\ref{sec:construction_of_an_integral_solution} we prove the existence and
uniqueness of an integral solution to \eqref{eq:pde_intro}, in particular the
solvability of \eqref{eq:global_intro}.  In Section
\ref{sec:a_Lyapunov_functional} we study the Lyapunov functional associated
to \eqref{eq:pde_intro}.  Section \ref{sec:examples} contains some initial
examples. Section \ref{sec:the_Riemann_problem} is essentially devoted to the
explicit construction of $\uM$ and the solution for the Riemann problem,
which are rather involved. The construction of $\uM$ is given in Theorem
\ref{teo:the_polygonal_is_the_graph_of_ell}, on the basis of the algorithm
described in \eqref{eq:slopes}, \eqref{eq:tmax_tmin_no_rearrangement}, and
\eqref{eq:tmax_tmin_rearrangement}.  How to directly recover,
a computable way, the integral
solution from the knowledge of $\uM$ is explained in Remark
\ref{rem:from_l_to_v}, see also Fig.  \ref{fig:riemann_specific_times}.  A
movie of the solution is illustrated in Figs. \ref{fig:Riemann_movie_I},
\ref{fig:Riemann_movie_II}. A few final examples are described 
in Section \ref{sec:final_examples}.

We conclude this introduction by remarking that the rigorous asymptotic
analysis of the discrete model as $h \to 0^+$ is out of the scope of the
present paper, and will be investigated elsewhere.

\section{Motivation: the discrete dynamical system and its rescaling}
\label{sec:the_discrete_game}
Let $\N$ be the nonnegative integers, and $\N^*$ be the positive
integers. Following \cite{Pao}, a configuration $\lambda$ is a map
$\lambda\colon\N^* \to\N$, and a mancala configuration is a configuration
$\lambda$ such that $\supp(\lambda):= \{\,i\,\in \N^*: \lambda_i >0\,\}$ is
connected, $\supp(\lambda)=\{1,2,\dots,{\rm len}(\lambda)\}$, with ${\rm
len}(\lambda) \in \N^*$ called the length of the configuration.  The {\it
mass} of a mancala configuration $\lambda$ is defined as
$|\lambda|:=\sum_{i=1}^{+\infty}\lambda_i
=\sum_{i=1}^{\len(\lambda)}\lambda_i$.  We denote with $\Lambda$ the set of
all mancala configurations, and with $\Lambda_\mass$ the set of all mancala
configurations with mass $\mass \geq 0$.

The {\it open mancala game} is a discrete dynamical system associated with a
function $\mathcal{M}\colon\Lambda\to\Lambda$; the action of $\M$ on a
mancala configuration $\lambda$ is defined by
\begin{equation}\label{eq:sowing}
\M(\lambda)_j:=\begin{cases}
\lambda_{j+1}+1 & \hbox{if} ~1\le j\le\lambda_1,\\
\lambda_{j+1}   & \hbox{if}~ j>\lambda_1.
\end{cases}
\end{equation}
Conventionally we also assume $\M(000\cdots)=000\cdots$.

For clarity, it is convenient to split $\M$ into two elementary
``half-moves'': the sowing move $\mathcal{S}$, which consists in taking all
seeds in the leftmost hole ($j=1$) and redistribute them between the next
holes by putting one seed for each hole until the seeds are over,
\[\mathcal{S}(\lambda)_j:=\begin{cases}
0               & \hbox{if}~ j=1,\\
\lambda_j+1   & \hbox{if}~ 2\le j\le\lambda_1+1,\\
\lambda_j     & \hbox{if}~ j>\lambda_1+1,
\end{cases}\]
and the left-shift move\footnote{This is done for convenience in order to obtain
a reference system following the configuration, and corresponds to a right
translation of the observer.  } $\mathcal{L}$, which shifts to the left all
holes by one,
\[\mathcal{L}(\lambda)_j:=\lambda_{j+1}, \qquad j \geq 1.\]
The composite map $\mathcal{L}\circ\mathcal{S}$ maps mancala configurations
into mancala configurations and indeed we have
$\mathcal{L}\circ\mathcal{S}(\lambda) =\M(\lambda)\colon\Lambda_\mass \to
\Lambda_\mass$.

It is natural to introduce an integer parameter $\kappa \geq 0$, standing for
the discrete time; given a configuration $\lambda^\kappa \in \Lambda_m$, the
sowing half-move is
\begin{equation}\label{eq:first_half_move}
\begin{cases}
\lambda_1^{\kappa+\frac{1}{2}} := 0 & {}\\
\lambda_j^{\kappa+ \frac{1}{2}}
:= \lambda_j^\kappa + 1 & {\rm if}~
2 \leq j \leq \lambda_1^\kappa + 1,\\
\lambda_j^{\kappa+ \frac{1}{2}}
:= \lambda_j^\kappa & {\rm if}~
j > \lambda_1^\kappa + 1,
\end{cases}
\end{equation}
and the left-shift half-move is
\[\lambda_j^{\kappa+ 1}:= \lambda_{j+1}^{\kappa+\frac{1}{2}}.\]
The composition of the two half-moves gives therefore the move
\[
\begin{cases}
\lambda_{j-1}^{\kappa+ 1}
= \lambda_j^\kappa + 1 & {\rm if}~
2 \leq j \leq \lambda_1^\kappa + 1,
\\
\lambda_{j-1}^{\kappa+ 1}
= \lambda_j^\kappa  & {\rm if}~
j > \lambda_1^\kappa + 1,
\end{cases}\]
i.e.,
\begin{equation}\label{eq:prima_formulazione_discreta}
\lambda_{j-1}^{\kappa+ 1}
= \lambda_j^\kappa + 
1_{I^\kappa(\lambda)}(j), \qquad 
{\rm where}~
I^\kappa(\lambda)
:=\{\,j\in \mathbb N:
2 \leq j\le\lambda^\kappa_1 + 1\,\},
\end{equation}
with $1_{I^\kappa(\lambda)}(j) = 1$ if $j \in I^\kappa(\lambda)$ and 
$0$ else.

\begin{Remark}[\textbf{Conservation of mass}]\rm
It is clear that mass is conserved:
\[\mass = \sum_{j=0}^{+\infty} \lambda_j^\kappa \qquad \forall \kappa \geq
0.\]
\end{Remark}

\begin{Remark}[\textbf{A discrete Lyapunov functional}]
\label{sec:a_discrete_Lyapunov_functional}\rm
In the discrete setting a Lyapunov functional, mimicking a gravitational
potential \cite{Pao}, can be defined as
\begin{equation}\label{eq:discrete_Lyapunov_functional}
\Lyap(\lambda) = \sum_{\substack{
(i,j) \in \N^* \times \N^*\\
i \leq j < i + \lambda_i}} j.
\end{equation}
\end{Remark}

Now, we suitably rescale the discrete dynamical system 
\eqref{eq:prima_formulazione_discreta}.

\subsection{Rescaling}\label{sec:rescaling_the_discrete_problem}
From \eqref{eq:prima_formulazione_discreta} it immediately follows
\begin{equation}
\label{eq:seconda_formulazione_discreta}
\lambda_{j}^{\kappa+ 1} -\lambda_{j}^{\kappa} 
= 
\lambda_{j}^{\kappa+ 1} - \lambda_{j-1}^{k+1}  
+ 1_{I^\kappa(\lambda)}(j).
\end{equation}

Now we rescale $j$, $\kappa$ and $\lambda$. If $0 < h \ll1$, we imagine a
grid $G_h$ in time-space $[0,+\infty) \times [0,+\infty)$ of the form $\{\,(kh,
ih) : k,h \in \N^*\,\}$. Next we substitute\footnote{Notice the hyperbolic
rescaling: time and space are rescaled the same way.}  in
\eqref{eq:seconda_formulazione_discreta} $kh$ in place of $\kappa$, $ih$ in
place of $j$, and we write $\lambda = u/h$.  The constraint $j \in
I^\kappa(\lambda)$ is transformed into
\[ih \in \{\,n h \in \N h : 2h \leq nh \leq (\lambda_1^\kappa+1)h\,\}
= \{\,n h \in \N h: 2h \leq n h \leq u_h^{kh}+h\,\}.
\]
We get from \eqref{eq:seconda_formulazione_discreta}
\begin{equation}\label{eq:discrete_equation_h}
\frac{u_{ih}^{(k+1)h}-u_{ih}^{kh}}{h} 
= 
\frac{u_{ih}^{(k+1)h}-u_{(i-1)h}^{(k+1)h}}{h}
+ 
1_{\left\{2h\leq\;\cdot\;\leq u_{h}^{kh} + h\right\}}(ih)
\end{equation}
for $i \geq 1$, $k \geq 0$.
We couple this rescaled dynamical system with the initial condition
\begin{equation}\label{eq:inidat_discrete_rescaled}
u_{ih}^0 = \inidat(ih), \qquad i \geq 1.
\end{equation}
Our aim is to study a continuous version of \eqref{eq:discrete_equation_h},
\eqref{eq:inidat_discrete_rescaled}.  If we interpret the quantities
$u^{kh}_{ih}$ as the values that a map $(t,x)\in \oic\times\oic \mapsto
u(t,x)$ takes on the points of the grid $G_h$ (i.e. $u^{kh}_{ih}=u(kh,ih)$)
then the equation \eqref{eq:discrete_equation_h} suggests that $u$ formally
satisfies the PDE in \eqref{eq:pde_intro}.

\begin{Remark}\rm
We have
\[\mass = \frac{1}{h} \sum_{i=1}^{+\infty} u_{ih}^{kh},\]
and the series is a finite sum.
\end{Remark}

\section{The continuous problem: a singular transport equation}
\label{sec:the_continuous_problem:t_x_formulation} 
$1_A$ stands for the characteristic function of the set $A \subset \R^k$,
$k=1,2$, i.e., $1_A(x)=0$ if $x \in A$ and $1_A(x)=0$ if $x
\in\R^k \setminus A$.
  $\vert \cdot \vert$ denotes the Lebesgue measure in $\R$. All
functions we consider are nonnegative and Lebesgue measurable.

We denote by $\BV_{\rm loc}([0,+\infty))$ the class of all functions $v$ of
finite pointwise variation in $[0,a]$, for any $a >0$.  We recall \cite{AmFuPa:00}
that $v$ is
bounded in $[0,a]$, there exists finite $\lim_{x \to 0^+} v(x) =: v(0)$,
and $v$ admits finite right and left limits at any $x \in
(0,+\infty)$.  We also let \cite{GiMoSo:98}
\begin{equation}\label{eq:subgraphs}
 {\rm subgr}_+(v) := \{\,(t,x) \in \oi \times \oi: v(t) > x\,\}
\end{equation}
be the subgraph of $v$ in $\oi \times \oi$; its reduced boundary in
$(0,+\infty) \times (0,+\infty)$ is given by the generalized graph of $v$,
i.e., the graph of $v$ with the addition of vertical segments joining the
left and the right limits at the jump points.

\begin{Definition}[\textbf{Polygonal function}]\label{def:piecewise_affine}
We say that $v\in \BV_{{\rm loc}}([0,+\infty))$ is polygonal if its generalized graph consists, in any compact
set $K \subset [0,+\infty)$, of a finite
number of segments (vertical segments are allowed\footnote{$v$ can be
discontinuous, with a finite number of jump points in $K$.}).
\end{Definition}

Recall that if $v\colon[0,+\infty)\to[0,+\infty)$, then the function $\lambda
\in \oi \mapsto \vert \{\,x \in [0,+\infty) : v(x) > \lambda\,\}\vert$ is
nonincreasing and right-continuous.

We start the study of the PDE in \eqref{eq:pde_intro}; from now on, we always
assume:
\begin{equation}\label{eq:inidat_from_now_on}
\inidat \in \BV_{\rm loc}([0,+\infty)), 
\qquad \inidat  \hbox{ right continuous and nonnegative}.
\end{equation}

\subsection{$(\tau,\xi)$-formulation}
Duly motivated by \eqref{eq:discrete_equation_h} we consider problem
\eqref{eq:pde_intro}, for which we need a rigorous notion of solution.  Let
be given a nonnegative function $v$
defined everywhere on  $\firstquadrantopen$, and suppose that for any $t > 0$ there exists $\lim_{x
\to 0^+} v(t,x) =: v(t,0) \in \oi$.

\begin{Definition}[\textbf{Active and tail regions; free
boundary}]\label{def:active_and_tail_regions} We call
\[\activereg(v):=\{\,(t,x) : t> 0,\; 0<x < v(t,0)\,\},
\qquad \tailreg(v) := \{\, (t,x) : t> 0,\; x > v(t,0)\,\},\]
the active region\footnote{The active region coincides with ${\rm
subgraph}_+(v(\cdot,0))$.} and the tail region (of $v$), respectively.  We
also define the free boundary $\freebound(v)$ as
\begin{equation}
\label{eq:Gamma_u}
\freebound(v)
:= \{\,(t,x)\in \firstquadrantopen : x = v(t,0)\,\}.
\end{equation}
\end{Definition}

The problem that we want to study reads {\it formally} as
\begin{equation}\label{eq:limit_problem_x}
\begin{cases}
u_t(t,x) = u_x (t,x) + 1_{(0,u(t,0))}(x) 
& {\rm if}~(t,x) \in (0,+\infty) \times (0,+\infty),\\
u(0,x) = \inidat(x) & {\rm if}~ x \in \oi.
\end{cases}
\end{equation}

Observe that 
\[1_{(0,u(t,0))}(x) =
1_{(x,+\infty)}(u(t,0)) \qquad \forall x > 0\]
and that the PDE reads as 
\[u_t(t,x) = 
\begin{cases}
u_x (t,x)  & {\rm if~} x > u(t,0), 
~{\rm i.e.},~ (t,x) \in T(u), \\
u_x (t,x) + 1 & {\rm if~} x < u(t,0),
~{\rm i.e.},~ (t,x) \in A(u).
\end{cases}\]

A natural sense in which we can interpret \eqref{eq:limit_problem_x} is
obtained using transformation $\Phi$ in
\eqref{eq:linear_change_of_variables}.  Given a function $v\colon
\firstquadrantopen \to [0,+\infty)$, define
\begin{equation}\label{eq:change}
\widehat v(\tau,\xi) := v(\Phi^{-1}(\tau,\xi)) = v(t,x)
\qquad \forall (\tau, \xi) \in \secondoctantopen.
\end{equation}
Then formally the PDE in \eqref{eq:limit_problem_x} transforms into
\begin{equation}
\label{eq:limit_problem_xi}
\begin{cases}
\wu_\tau(\tau, \xi) =  
1_{(0,\wu(\tau,\tau))}(\xi-\tau) =
1_{(\tau, \wu(\tau,\tau) + \tau)}(\xi)
& {\rm if}~ (\tau,\xi) \in \soctopen,\\
\wu(0,\xi) = \inidat(\xi) & {\rm if}~ \xi \in \oi.
\end{cases}
\end{equation}
Active (resp. tail) region $\activereg(u)$ (resp. $\tailreg(u)$) transforms
into
\[\activereg(\wu) 
= \{\,(\tau,\xi) : \tau > 0,\; \tau < \xi < \wu(\tau,\tau) + \tau\,\},
\quad
\tailreg(\wu) = \{\,(\tau,\xi) : \tau > 0,\; \xi > \wu(\tau,\tau)+\tau\,\},\]
and the set $\freebound(u)$ in \eqref{eq:Gamma_u} transforms into
\begin{equation}\label{eq:Gamma_u_nuove_variabili}
\freebound(\wu) = \{\,(\tau,\xi) : \xi > \tau,\; \xi =\wu(\tau,\tau)+\tau
\,\}.
\end{equation}

We are now in a position to define what we mean by 
a solution of \eqref{eq:limit_problem_x}.

\begin{Definition}[\textbf{Integral solution}]\label{def:integral_solution}
We say that a locally bounded function $\wu\colon\soctopen\to\oi$ defined
everywhere is an integral solution to \eqref{eq:pde_intro} if the following
conditions hold:
\begin{itemize}
\item[(i)]  for any $\tau >0$ 
\[\exists \lim_{\xi\to \tau^+} \wu(\tau, \xi) =: \wu(\tau,\tau) \in \oi;\]
\item[(ii)] 
for any $(\tau, \xi) \in \soctopen$, 
\begin{equation}\label{eq:equazione_integrale_0}
\wu(\tau,\xi) = \inidat(\xi) + \int_0^\tau 1_{(s, \wu(s,s) + s)}(\xi)\,ds.
\end{equation}
\end{itemize}
\end{Definition}

Some comments are in order.
\begin{itemize}
\item[(a)] if $\wu$ is an integral solution then, for any $\xi >0$, the
function $\tau \in (0,\xi) \mapsto \wu(\tau, \xi)$ is nondecreasing and
one-Lipschitz.
\item[(b)] For any $(\tau, \xi) \in \secondoctantopen$ we have\footnote{In the 
first equality of \eqref{eq:so_that} we 
use that, 
if $f\colon \oi \to \oi$ is measurable and $I \subseteq \oi$,
then 
$\int_I f ds = \int_{(0,+\infty)} \vert I \cap \{f > \lambda\}\vert\,
d\lambda$.}
\begin{equation}\label{eq:so_that}
\begin{aligned}
&\int_0^\tau 1_{(s, \wu(s,s) + s)}(\xi)\, ds
= \int_0^1 \vert \{\, s \in [0,\tau] : 1_{(s, \wu(s,s) + s)}(\xi)>\lambda
\,\}
\vert\,d\lambda \\
=& \int_0^1\vert \{\, s \in [0,\tau] : s <\xi < \wu(s,s) + s\,\}
\vert\,d\lambda = 
\vert\{\,s \in [0,\tau] : 
s < \xi < \wu(s,s) + s \,\}\vert \\
= & \vert \{\,s \in [0,\tau] : \wu(s,s) + s> \xi\,\}\vert,
\end{aligned}
\end{equation}
so that \eqref{eq:equazione_integrale_0}
is equivalent to 
\begin{equation}\label{eq:integral_is_equivalent}
\wu(\tau,\xi) = \inidat(\xi) +\vert\{\,s \in [0,\tau] : 
\wu(s,s) + s  >  \xi\,\}\vert.
\end{equation}
In particular, since $\inidat$ is right-continuous, passing to the limit as
$\xi \to \tau^+$ and using (i),
\begin{equation}\label{eq:mistery} 
\wu(\tau,\tau) = \inidat(\tau) + 
 \vert \{\, s \in [0,\tau] : \wu(s,s) + s
 >  \tau \,\}\vert 
\end{equation}
for all $\tau \in (0,+\infty)$.
\end{itemize}

We conclude this section with a notation: given $\overline t>0$, we let
$\halfline_{\overline t} =\{\,(t,\incrt) : \incrt \geq 0,\; \incrt=h_{\overline
t}(t)\,\}$ be the half-line pointing up-left at $45^o$ and passing through $(\overline
t,0)$, defined by $h_{\overline t}(t) = \overline t-t$.

\section{Solving the equation for the free
boundary}\label{sec:another_formulation}
In this section we solve an auxiliary problem (see Theorem
\ref{teo:existence_of_l}) which, in view of the results in Section
\ref{sec:construction_of_an_integral_solution}, is essentially equivalent to
find an integral solution of \eqref{eq:pde_intro}: the idea is to look for an
expression of what, a posteriori, will be the curve $\freebound(\wu)$ and
then to reconstruct $\wu$ itself (and hence $u$) by the method of
characteristics, distinguishing the active and the tail regions; here a
difficulty arises since characteristic lines may pass from the active region
to the tail region several times.  Recall our assumption
\eqref{eq:inidat_from_now_on} on $\inidat$, and keep in mind
\eqref{eq:mistery}.

\begin{Definition}[\textbf{The function
$\unknownMau$}]\label{def:the_function_l}
We say that the graph of a function $\unknownMau = \unknownMau[\inidat]\colon
\oi \to \oi$ represents the free boundary of problem \eqref{eq:pde_intro} if
\begin{equation}
\label{eq:how_to_reconstruct_l_from_itself}
\unknownMau(t) = \inidat(t) + \vert\{\,\incrt \in [0,t] : 
\unknownMau(t-\incrt) >\incrt \,\}\vert 
\qquad \forall t \in \oi.
\end{equation}
\end{Definition}

In what follows, we shall frequently use the obvious equality
\begin{equation}\label{eq:obv}
\vert\{\,\incrt \in [0,t] : \unknownMau(t-\incrt) 
>\incrt\,\}\vert = \vert\{\,\incrt \in [0,t] : \unknownMau(\incrt) 
+ \incrt > t\,\}\vert.
\end{equation}

Notice that, since $\inidat$ is right-continuous, 
also $\unknownMau$ is right-continuous; moreover
\[\unknownMau(0) = \inidat(0).\]

\begin{Remark}[\textbf{Geometric meaning}]\label{rem:geometric_meaning}\rm
Formula \eqref{eq:how_to_reconstruct_l_from_itself} has a geometric meaning:
recalling the notation in the end of Section
\ref{sec:the_continuous_problem:t_x_formulation}, consider the half-line
$\halfline_t=\{(t-\incrt,\incrt): \incrt > 0\}$.  Next, take the intersection
$I_t$ of $\halfline_t$ with ${\rm subgraph}_+(\uM)= \{(t,s) \in (0,+\infty)
\times (0,+\infty) : \uM(t)>s \}$, so that
\[I_t=\{\,(t,\incrt) : \incrt \in [0,t],\; \unknownMau(t-\incrt)>
\incrt\,\}.\]
Denote by $\pi_2\colon\oi \times \oi \to \{0\} \times \oi$ 
the orthogonal projection on the vertical $\incrt$-axis. 
Then 
\begin{equation}\label{eq:geometric_meaning}
\unknownMau(t) = \inidat(t) +
\vert \pi_2 (I_t)\vert, \qquad t \geq 0.  \end{equation}
In particular, if $\inidat$ is polygonal, then $\uM$ is 
polygonal\footnote{Recall that jumps are not excluded.}.
\end{Remark}

Proving that \eqref{eq:how_to_reconstruct_l_from_itself} has a solution is
nontrivial; in order to do that we start to show that {\it if
\eqref{eq:how_to_reconstruct_l_from_itself} has a solution} and $\inidat$ has
a suitable behaviour close to the origin, then $\unknownMau$ cannot go to
zero at any point of $[0,+\infty)$.
 
\begin{Proposition}[\textbf{Lower bound}]\label{prop:bound_of_l_from_below}
Suppose that $\uM$ solves \eqref{eq:how_to_reconstruct_l_from_itself} and 
furthermore
\begin{equation}\label{eq:inidat_larger_than_alpha}
\exists \alpha>0: \inidat(t)\geq \alpha \qquad \forall t \in [0,\alpha).
\end{equation}
Then 
\begin{equation}\label{eq:l_larger_than_alpha}
\unknownMau(t) \geq \alpha \qquad \forall t \in \oi.
\end{equation}
\end{Proposition}

Assumption \eqref{eq:inidat_larger_than_alpha} says that we can pick a square
$Q=(0,\alpha) \times (0,\alpha)$ contained in ${\rm subgraph}_+(\uM)$, and
\eqref{eq:l_larger_than_alpha} says that we can slide $Q$ horizontally
remaining inside ${\rm subgraph}_+(\uM)$, and this is crucial in what
follows; dropping this assumption is discussed in Theorem
\ref{teo:if_we_drop} and in Section
\ref{subsec:initial_condition_vanishing_at_the_origin}.

\begin{proof}
Suppose by contradiction that \eqref{eq:l_larger_than_alpha} 
is false, and define 
$t_0 := \inf \{\, t \in \oi: \unknownMau(t) < \alpha\,\}$.
From \eqref{eq:inidat_larger_than_alpha} and
\eqref{eq:how_to_reconstruct_l_from_itself} it follows
\begin{equation}\label{eq:t_0_alpha} 
t_0 \geq \alpha.
\end{equation}
Fix now $\eps \in (0,\alpha)$, that will be selected
later (see \eqref{eq:prov}). Pick $t \in (t_0, t_0 + \eps)$,
and write $t = t_0 + \eta$,
with $\eta \in (0,\eps)$. In particular, $\eta < \alpha$.
Set
\[\primoinsieme
:= \{\,\incrt \in [0,\eta) : \unknownMau(t-\incrt)
> \incrt\,\}, \qquad\secondoinsieme
:=\{\,\incrt \in [\eta,\alpha) : \unknownMau(t-\incrt)> 
\incrt\,\}.\]
From \eqref{eq:how_to_reconstruct_l_from_itself} and recalling that $\inidat$
is nonnegative it follows
\begin{equation}\label{eq:l_geq_AB}
\begin{aligned}
\unknownMau(t) &\geq \vert \{\,\incrt \in [0,t] : 
\unknownMau(t-\incrt)>\incrt \,\}\vert\\
&= \vert \primoinsieme\vert +\vert \{\,\incrt \in [\eta,t] : 
\unknownMau(t-\incrt)> \incrt\,\}\vert
 \geq  \vert \primoinsieme \vert + \vert \secondoinsieme\vert,
\end{aligned}
\end{equation}
where in the last inequality we have used \eqref{eq:t_0_alpha}, so that $t >
t_0 \geq \alpha$.

\medskip\noindent
{\it Claim 1}: We have $\vert \secondoinsieme\vert  = \alpha - \eta$.

Indeed, if $\incrt \in \secondoinsieme$,
in particular
$\incrt \geq \eta$, and so
$t-\incrt = t_0 + \eta - \incrt \leq t_0$.
Hence, recalling the definition of $t_0$, we have $\unknownMau(t-\incrt) \geq \alpha$.
But for $\incrt \in \secondoinsieme$ we have $\incrt < \alpha$, and 
so $\unknownMau(t-\incrt) > \incrt$, which implies
$\vert \secondoinsieme\vert = \alpha - \eta$.

{}From \eqref{eq:l_geq_AB}, claim 1 and $\eta \in (0,\eps)$ it follows
\begin{equation}\label{eq:intermediate}
\unknownMau(t) \geq \alpha - \eps \qquad \forall t \in (t_0, t_0 + \eps).
\end{equation}
\medskip

\noindent
{\it Claim 2}: If $\eps < \alpha/2$ then $\vert \primoinsieme\vert = \eta$.

We start to observe that for any $\incrt \in \primoinsieme$ we have $t_0 \leq
t-\incrt < t_0 + \eps$. Indeed, $t- \incrt = t_0 + \eta - \incrt < t_0
+\eps$ since $\eta-\incrt < \eps -\incrt \leq \eps$; on the other hand, $t-
\incrt = t_0 + \eta - \incrt > t_0$ since $\eta-\incrt > 0$, from the
definition of $\primoinsieme$.  Therefore we are allowed to replace $t$ with
$t-\incrt$ in \eqref{eq:intermediate}, and we get
\[\unknownMau(t-\incrt) \geq \alpha-\eps
\qquad
\forall \incrt \in \primoinsieme.\]
Due to our choice of $\eps$, we have $\alpha - \eps > \alpha/2$, and
$\alpha/2> \incrt$, since $\alpha/2 > \eps > \eta$. Hence all $\incrt \in
[0,\eta)$ are points of $\primoinsieme$, and the claim follows.

\medskip
From \eqref{eq:l_geq_AB} and claims 1 and 2 we deduce, provided
\begin{equation}\label{eq:prov}
\eps \in (0,\alpha/2),
\end{equation} 
that
\[\unknownMau(t) \geq \eta + \alpha - \eta = \alpha \qquad \forall t \in
(t_0, t_0 + \eps),\]
which contradicts the definition of $t_0$. 
\end{proof}

\begin{Lemma}\label{lem:transl}
Under assumption \eqref{eq:inidat_larger_than_alpha}, $\unknownMau$ is a
solution of \eqref{eq:how_to_reconstruct_l_from_itself} if and only if
\begin{equation}\label{eq:how_translated}
\unknownMau(t) = 
\begin{cases}
\inidat(t) + t   \qquad & \forall t \in [0,\alpha),\\
\inidat(t) + \alpha + 
\vert\{\,\incrt \in [\alpha,t] : 
\unknownMau(t-\incrt) >\incrt\,\}\vert \qquad & \forall t \in [\alpha,+\infty) .
\end{cases}
\end{equation}
%
\end{Lemma}

\begin{proof}
Suppose that $\unknownMau$ satisfies
\eqref{eq:how_to_reconstruct_l_from_itself}.  According to
Proposition~\ref{prop:bound_of_l_from_below}, $\unknownMau$ must satisfy
\eqref{eq:l_larger_than_alpha}. Hence, for any $t \in [\alpha, +\infty)$,
\[\begin{aligned}
\unknownMau(t) &= \inidat(t) + 
\vert\{\,\incrt \in [0,\alpha) : \unknownMau(t-\incrt) >\incrt\,\}\vert 
+\vert\{\,\incrt \in [\alpha,t] : \unknownMau(t-\incrt)>\incrt\,\}\vert \\
&=\inidat(t) + \alpha+ \vert\{\,\incrt \in [\alpha,t] : \unknownMau(t-\incrt) 
>\incrt \,\}\vert 
\end{aligned}\]
since, from \eqref{eq:how_to_reconstruct_l_from_itself}, if $\incrt \in
[0,\alpha)$, then $\unknownMau(t-\incrt) \geq \alpha >\incrt$.  The case $t
\in [0, \alpha)$ is a direct consequence of \eqref{eq:l_larger_than_alpha}
and \eqref{eq:how_to_reconstruct_l_from_itself}.

The proof of the converse implication is similar.
\end{proof}

\begin{Theorem}[\textbf{Global existence and uniqueness of
$\unknownMau$}]\label{teo:existence_of_l}
Let $\inidat$ be as in \eqref{eq:inidat_from_now_on}, and suppose that
\eqref{eq:inidat_larger_than_alpha} holds. Then
\eqref{eq:how_to_reconstruct_l_from_itself} admits a unique solution
$\unknownMau \in \BV_{{\rm loc}}(\oi)$. Furthermore
$\uM$ satisfies \eqref{eq:l_larger_than_alpha}, 
$\unknownMau(t) \leq 
\Vert \inidat\Vert_\infty 
 + t$
for all $t \in \oi$, 
and
\begin{equation}\label{eq:l_minus_id_is_decreasing}
\unknownMau(t+\tau)-\unknownMau(t)
\leq 
\inidat(t+\tau)-\inidat(t) + \tau, \qquad t > 0,\; \tau \geq 0.
\end{equation}
\end{Theorem}

In particular, suppose ${\rm spt}(\inidat) \subset [0,a]$ and $t > a$;
then $\unknownMau$ cannot have increasing jumps, as well as increasing Cantor
parts, in $(a,+\infty)$.

\begin{proof}
We construct $\unknownMau$ as follows: first we set 
\begin{equation}\label{eq:first_step}
\unknownMau(t) := \inidat(t) + t \qquad \forall t \in [0,\alpha].
\end{equation}
Next, keeping in mind \eqref{eq:how_translated}, we define
\begin{equation}\label{eq:second_step}
\unknownMau(t) := \inidat(t) + \alpha + 
\vert\{\,\incrt \in [\alpha,t] : 
\unknownMau(t-\incrt) 
>\incrt\,\}\vert \qquad \forall t \in [\alpha,2\alpha].
\end{equation}
Note that $\unknownMau$ is well-defined in the interval $[\alpha,2\alpha]$,
since if $t \in [\alpha, 2\alpha]$ and $\incrt \in [\alpha,t]$, then 
$t-\incrt \in [0,\alpha]$, and we can use \eqref{eq:first_step}. 

We now repeat the argument inductively for $t \in [k\alpha, (k+1)\alpha]$ for
any integer $k \geq 2$; for instance, if $k=2$, for any $t \in [2\alpha,
3\alpha]$ we set
\begin{equation}\label{eq:third_step}
\begin{aligned}
\unknownMau(t) :=& 
\inidat(t) + \alpha + 
\vert\{\,\incrt \in [\alpha,t] : 
\unknownMau(t-\incrt) >\incrt \,\}\vert \\
=&
\inidat(t) + \alpha + 
\vert
\{\,\incrt \in [\alpha,2\alpha] : 
\unknownMau(t-\incrt) 
>
\incrt\} 
\vert
\\
& +
\vert
\{\,\incrt \in (2\alpha,t] : 
\unknownMau(t-\incrt)
>
\incrt 
\,\} 
\vert,
\end{aligned}
\end{equation}
where we notice that if $\incrt \in [\alpha,2\alpha]$ then $t-s \in
[0,2\alpha]$, so that
\[\begin{aligned}
\vert
\{\,\incrt \in [\alpha,2\alpha] : 
\unknownMau(t-\incrt)
>
\incrt \,
\} \vert
=&
\vert
\{\,\incrt \in [\alpha,2\alpha] : t-\incrt \in [0, \alpha], 
\unknownMau(t-\incrt)
>
\incrt\,\} \vert\\
& + \vert
\{\,\incrt \in [\alpha,2\alpha] : t-\incrt \in (\alpha, 2\alpha], 
\unknownMau(t-\incrt)>\incrt \,\} \vert,
\end{aligned}\]
and if $\incrt \in [2\alpha,t]$ then
$t-s \in [0,\alpha]$; thus \eqref{eq:third_step} is well-defined by 
\eqref{eq:first_step}
 and
\eqref{eq:second_step}\footnote{For any $t \in [k\alpha, (k+1)\alpha]$ we set 
\begin{equation}\label{eq:third_stepk}
\begin{aligned}
\unknownMau(t) :=& 
\inidat(t) + \alpha + 
\vert
\{\,\incrt \in [\alpha,t] : 
\unknownMau(t-\incrt) >\incrt \,\}\vert 
\\
=&
\inidat(t) + \alpha + 
\sum_{j=1}^{k-1} \vert 
\{\,\incrt \in [j\alpha,(j+1)\alpha] : 
\unknownMau(t-\incrt)\, >\incrt \}\vert
\\
& +\vert\{\,\incrt \in [k\alpha,t] : 
\unknownMau(t-\incrt)>\incrt\,\} \vert,
\end{aligned}
\end{equation}
where, if 
$\incrt \in [j\alpha,(j+1)\alpha]$ then
$t-s \in [(k-1-j)\alpha,(k+1-j)\alpha]$, so that 
\[
\begin{aligned}
\vert
\{\,\incrt \in [j\alpha,(j+1)\alpha] : 
\unknownMau(t-\incrt)>\incrt \,\} \vert
=&
\vert\{\,\incrt \in [j\alpha,(j+1)\alpha] : t-\incrt \in [(k-1-j)\alpha,
(k-j)\alpha], \unknownMau(t-\incrt)>\incrt\,\} \vert
\\
& + \vert\{\,\incrt \in [j\alpha,(j+1)\alpha] : t-\incrt \in ((k-j)\alpha,
(k-1+j)\alpha],\unknownMau(t-\incrt)>
\incrt \,\} \vert,
\end{aligned}\]
and if $\incrt \in [2k\alpha,t]$ then
$t-s \in [0,\alpha]$; thus \eqref{eq:third_step} is well-defined by recursion.
}.

In this way we construct a globally defined function $\uM:\oi \to
(0,+\infty)$ satisfying \eqref{eq:how_translated} and
\eqref{eq:l_larger_than_alpha}.  Hence the existence of a solution to
\eqref{eq:how_to_reconstruct_l_from_itself} follows from Lemma
\ref{lem:transl}.  To prove that $\unknownMau$ is unique, it is sufficient to
show that it is unique in $[0,\alpha)$; since $\unknownMau \geq \inidat \geq
\alpha$ in $[0,\alpha]$, it follows from \eqref{eq:second_step} that
$\unknownMau(t) = \inidat(t) + t$ for any $t \in [0,\alpha)$, as desired.

From \eqref{eq:first_step} and the assumption \eqref{eq:inidat_from_now_on}
it follows $\unknownMau \in \BV([0,\alpha])$; then the $\BV_{\rm
loc}$-regularity of $\unknownMau$ follows observing that the
term $\vert\{\,\incrt \in [0,t] : \unknownMau(t-\incrt) > \incrt\, \}\vert$
in \eqref{eq:how_to_reconstruct_l_from_itself} can be written as difference
of two nondecreasing functions. More specifically, using also \eqref{eq:obv},
\[\begin{aligned}
& \vert\{\,\incrt \in [0,t] : 
\unknownMau(t-\incrt) > \incrt\,\}\vert 
= t - \vert\{\,\incrt \in [0,t] : 
\unknownMau(t-\incrt) 
\leq
\incrt  
\,\}\vert
\\
=& 
t- \vert\{\,\tau \in [0,t] : 
\unknownMau(\tau) + \tau 
\leq
t  
\,\}\vert,
\end{aligned}\]
and the last term on the right-hand side is nondecreasing in $t$.  Clearly,
the bound
$\unknownMau(t) \leq 
\Vert \inidat\Vert_\infty 
 + t$
for all $t \in \oi$,
follows from
\eqref{eq:how_to_reconstruct_l_from_itself}.

It remains to prove that $\unknownMau$ satisfies the one-sided Lipschitz
condition \eqref{eq:l_minus_id_is_decreasing}.  Let $\overline t>0$ and
$\tau>0$; since
\[\{\,t \in [0,\overline t]: h_{\overline t +\tau}(t) \leq \unknownMau(t)\,\} 
\subseteq 
\{\,t \in [0,\overline t]: h_{\overline t}(t) < \unknownMau(t)\,\},\]
we have 
v\[\begin{aligned}
& \vert \{\,t \in [0,\overline t]: 
h_{\overline t +\tau}(t) \leq \unknownMau(t)\,\} 
\cup \{\,t \in (\overline t, \overline t+\tau] : 
h_{\overline t +\tau}(t) \leq \unknownMau(t)\,\}\vert\\ 
 = & 
\vert \{\,t \in [0,\overline t]: 
h_{\overline t +\tau}(t) \leq \unknownMau(t)
\,\} \vert + 
\vert \{\,t \in (\overline t, \overline t+\tau] : 
h_{\overline t +\tau}(t) \leq \unknownMau(t)\,\}\vert\\
\leq & 
\vert \{\,t \in [0,\overline t]: 
h_{\overline t }(t) < \unknownMau(t)
\,\} \vert + 
\vert \{\,t \in (\overline t, \overline t+\tau] : 
h_{\overline t +\tau}(t) \leq \unknownMau(t)\,\}\vert
\\
\leq & 
\vert \{\,t \in [0,\overline t]: 
h_{\overline t }(t) < \unknownMau(t)
\,\} \vert + 
\tau.  \end{aligned}\]
Whence
\eqref{eq:l_minus_id_is_decreasing} follows Remark
\ref{rem:geometric_meaning}, since 
the Lebesgue measures of the ortoghonal projection of 
$I_{\overline t+\tau}$ and
$I_{\overline t}$ 
on the $t$-axis, are the same as the Lebesgue measure of the
corresponding projections on the $s$-axis.
\end{proof}

\begin{Remark}[\textbf{Shorter steps}]\label{rem:piecewise_affine}\rm
Suppose \eqref{eq:inidat_larger_than_alpha}, and take $\beta \in (0,\alpha)$;
in particular $\uM(t) \geq \beta$ for any $t \in [0,\beta]$. Then the
constructive proof in Theorem \ref{teo:existence_of_l} (see 
\eqref{eq:second_step}, \eqref{eq:third_step}, \eqref{eq:third_stepk}) obtained replacing
$\alpha$ with $\beta$ leads to the same $\unknownMau$. 
\end{Remark}

The following useful result has a straightforward proof.

\begin{Lemma}[\textbf{Monotonicity}]\label{lem:monotonicity}
\rm
Suppose that $\inidat_1, \inidat_2$ are two nonnegative functions in
$\BV_{{\rm loc}}(\oi)$, both satisfying
\eqref{eq:inidat_larger_than_alpha}. Let $\unknownMau_1, \unknownMau_2$ be
the corresponding solutions given by Theorem \ref{teo:existence_of_l}.  Then
\[\inidat_1 \leq \inidat_2 \Longrightarrow 
\unknownMau_1 \leq \unknownMau_2.\]
\end{Lemma}
\begin{proof}
Let $\alpha>0$ be such that \eqref{eq:inidat_larger_than_alpha} holds both
for $\inidat_1$ and $\inidat_2$.
From our assumption and 
\eqref{eq:how_translated} it immediately follows that 
$\unknownMau_1 \leq \unknownMau_2$ in $[0,\alpha]$. Then 
the same inequality holds for $t \in [\alpha, 2\alpha]$, as 
a consequence of the inclusion 
$\{\,\incrt \in [\alpha,t] : \unknownMau_1(t-\incrt)
>\incrt \,\} 
\subseteq
\{\,\incrt \in [\alpha,t] : \unknownMau_2(t-\incrt)>
\incrt\,\} $
and \eqref{eq:how_translated}. Then the assertion follows,
recalling the recurrence proof of Theorem \ref{teo:existence_of_l}.
\end{proof}

We conclude this section with a crucial definition
which, as we shall see, is related to the vertical rearrangement.

\begin{Definition}[\textbf{Critical segments and critical times}]
\label{def:critical_segments_and_critical_times}
Any segment in the graph of $\uM$ having slope in $[-\infty, -1)$ is called a
critical segment, and the corresponding slope is called a critical slope. The
$t$-coordinate of the left extremum of a critical segment is called a
critical time\footnote{If the segment has slope $-\infty$, its projection on
the $t$-axis is one point, still called a critical time.}.
\end{Definition}

Note carefully that any critical segment in the graph of $\uM$ forces the
existence of a range of $t>0$ for which $\halfline_t \cap {\rm
subgraph}_+(\uM)$ is {\it not connected}.

\subsection{Transformation $\affinetransformation$ and rearrangement}
\label{subsec:transformation_R_and_rearrangement}
In this short section we deepen the geometric meaning of formula 
\eqref{eq:geometric_meaning}; this shades light on the 
meaning of $\uM$, and is useful to construct examples (see
Sections \ref{sec:examples} and \ref{sec:the_Riemann_problem}). 

Let us introduce
the affine transformation
\begin{equation}\label{eq:R}
\affinetransformation\colon\firstquadrant \to \firstquadrant, 
\qquad \affinetransformation
(t,\incrt) := (t+\incrt,\incrt) .
\end{equation}
It moves a point to the right (same $\incrt$, larger $t$) of an amount equal
to its vertical $\incrt$-component\footnote{In particular, it transforms a
square $((k-1)\alpha, k\alpha) \times (0,\alpha)$ into the square $(k\alpha,
(k+1)\alpha) \times (0,\alpha)$, $k \in \mathbb N$.}.  Then there is a
peculiar relation between ${\rm subgraph}_+(\uM)$ (which is a locally finite
perimeter set \cite{GiMoSo:98}) and its $\affinetransformation$-transformed: the former is a
kind of rearrangement of the latter.

Specifically, consider the vertical rearrangement $F^\sharp$ of a locally
finite perimeter set $F \subset \oi \times \oi$, defined as follows: for any
$\overline t \in \oi$, if $F_{\overline t} := \{t=\overline t\} \cap F$, and
$F_{\overline t}^\sharp := \{\overline t\} \times [0, \vert F_{\overline
t}\vert]$, then
\[F^\sharp := \bigcup_{t \geq 0} F_t^\sharp,\]
which is still \cite[Theorem 14.4]{Mag} a locally finite perimeter subset of
$\oi \times \oi$, having the same Lebesgue measure as $F$. Suppose for
svimplicity that $\inidat$ has compact support, say ${\rm spt}(\inidat)
\subset [0,\support]$.  Then the validity of \eqref{eq:geometric_meaning} for
all $t \in [\support, +\infty)$ can be equivalently expressed as
\begin{equation}\label{eq:trasformato_riarrangiato}
{\rm subgraph}_+(\uM) 
\cap\Big( [\support, +\infty) \times \oi\Big) 
= \affinetransformation({\rm subgraph}_+(\uM))^\sharp
\cap\Big( [\support, +\infty) \times \oi\Big).
\end{equation}
Then it is immediately seen from \eqref{eq:trasformato_riarrangiato} that,
provided $k \alpha \geq \support$, we have
\begin{equation}\label{eq:oscillation}
\begin{aligned}
\sup_{t \in [(k+1)\alpha, (k+2)\alpha]} \uM(t)
\leq &
\sup_{t \in 
[k\alpha, (k+1)\alpha]} 
\uM(t) 
\\
\inf_{t \in [(k+1)\alpha, (k+2)\alpha]} \uM(t)
\geq &
\inf_{t \in [k\alpha, (k+1)\alpha]} \uM(t) 
\end{aligned}
\end{equation}
so that the oscillation of $\uM$ in $[k\alpha, (k+1)\alpha]$ 
is nonincreasing in $k$.

\nada{
\begin{figure}
\[\<<2>>\buildrel \affinetransformation\over \longrightarrow\,\<<3>>
\buildrel (\cdot)^\sharp \over \longrightarrow\,\<<4>>\]
\caption{Left: the initial (generalized) graph of $\uM$ and
its subgraph.
Center: the image through the map $\affinetransformation$
of the subgraph; it results in a set enclosed between
the horizontal axis and 
the graph of a multivalued function.
Right: result of the vertical rearrangement.
}\label{fig:affine_rearrangement}
\end{figure}
}

It is interesting to observe that the strict inequalities hold in
\eqref{eq:oscillation} under certain conditions on $\inidat$: for instance,
suppose that $\inidat$, beside having support contained in $[0,\support]$, is
polygonal. If there is a critical segment in the generalized graph of $\uM$
whose projection on the horizontal $t$-axis is contained in $[k\alpha,
(k+1)\alpha]$, then the strict inequalities hold in \eqref{eq:oscillation}.
Finally, segments with slope $-1$ are transformed via $\afftransf$ to
vertical segments, and a critical segment, together with its left-adjacent,
are transformed into a polygonal which is not a graph with respect to $t$.
This construction will be empolyed in the examples, in particular 
the Riemann problem considered in Section
\ref{sec:the_Riemann_problem}.

\subsection{The functions $\unknownplustMau_\uM$ and $\inverseMau_\uM$}
It is convenient to introduce the following functions, that have been
implicitely used in the end of the proof of Theorem \ref{teo:existence_of_l}.

\begin{Definition}[\textbf{$L_\uM$ and $S_\uM$}]\label{def:the_function_S}
Let $\inidat$ and $\uM = \uM[\inidat]$ be as in Definition \ref{def:the_function_l}.
We set
\[\unknownplustMau_\uM(t) := 
\unknownMau(t) + t \qquad \forall t \geq 0\]
and
\begin{equation}\label{eq:inverseMau}
\inverseMau_\uM(\incrt) := 
\begin{cases}
0 & {\rm if~} \incrt \in [0, \uM(0)],
\\
|\{\,\sigma \in [0,\incrt] : 
\unknownplustMau_\uM(\sigma) 
\leq
 \incrt\,\}| 
& {\rm if~} \incrt > \uM(0).
\end{cases}
\end{equation}
\end{Definition}

\begin{Remark}\label{rem:the_function_S}\rm
The function $\inverseMau_\uM$ is nondecreasing in $\oi$, and for any $t>0$
\begin{equation}
\label{eq:how_to_reconstruct_l}
t-\inverseMau_\uM(t) = 
\vert\{\,s \in [0,t] : \unknownMau(s) 
> t-s\,\}\vert= 
\vert\{\,\incrt \in [0,t] : \unknownMau(t-\incrt)
> \incrt\,\}\vert.
\end{equation}
Hence formula \eqref{eq:how_to_reconstruct_l_from_itself}
is equivalent to 
\begin{equation}\label{eq:b_B_u_0}
\unknownMau(t) =  \inidat(t)  + 
t-  \inverseMau_\uM(t) 
\qquad \forall t \in \oi.
\end{equation}
\end{Remark}
Before concluding this section, 
we summarize some of the obtained results.

\begin{Remark}\label{rem:summa}\rm 
The function $\uM$ can be described in various equivalent ways:
\begin{itemize}
\item[(i)] algebraically, using equation 
\eqref{eq:how_to_reconstruct_l_from_itself} or, equivalently, 
equation \eqref{eq:how_translated} or also equation \eqref{eq:b_B_u_0}; 
\item[(ii)] geometrically, computing the Lebesgue measure of the
projection on one of the coordinate axes
of the intersection of the up-left half-lines at $45^o$ 
with ${\rm subgraph}_+(\uM)$ (see formula \eqref{eq:geometric_meaning}), or using the transformation $\affinetransformation$
and the vertical rearrangement,
as explained in 
Section
\ref{subsec:transformation_R_and_rearrangement}.
\end{itemize}
\end{Remark}

\subsection{Examples}
It is worth showing some explicit computations of $\uM$;
more involved and interesting examples will be illustrated in Sections
\ref{sec:examples} and \ref{sec:the_Riemann_problem}.
The next example is particularly simple, and
concerns an initial condition not with
compact support; curiously enough, this example involves the triangular
numbers
\[T_k := \sum_{i=0}^k k = \frac{k(k+1)}{2}.\]

\begin{Example}[\textbf{Constant initial condition, I}]\label{exa:costante_uno}\rm
Let
\[u_0=1_{[0,+\infty)},\]
which has infinite mass. 
{}From \eqref{eq:first_step} and
\eqref{eq:second_step} one checks
that $\uM\colon\oi \to [1,+\infty)$ is
the Lipschitz piecewise affine function that interpolates the points
$(T_k, k+1)$, so that
\begin{equation}
\uM(t)=\frac{t}{k+1} + \frac{k}{2} +1, \qquad t\in[T_k,T_{k+1}], \ k \in
\mathbb N,
\label{eq:exa_l_1}
\end{equation} 
see Fig.~\ref{fig:characteristic_half_line}. Hence $L_\uM\colon\oi \to
[1,+\infty)$ is also Lipschitz piecewise affine, $L_\uM(T_k) = T_{k+1}$, and
reads as
\begin{equation}
\label{eq:exa_L_1}
L_\uM(t)=\frac{k+2}{k+1}(t-T_k)+T_{k+1},
\quad t\in[T_k,T_{k+1}], \quad k \in \mathbb N.
\end{equation}
Functions $\uM$ and $L_\uM$ are strictly increasing and surjective.

Equivalently,
$\uM$ can be constructed
 recursively using $S_\uM$ as follows.
Since $\alpha=1$,
from \eqref{eq:first_step} we obtain $\uM(t) = 1+t$ (and $L_\uM(t)= 1+2t$)
for $t \in [0,1]$; from \eqref{eq:inverseMau} we have $S_\uM=0$ in $[0,1]$,
and $S_\uM(s)= (s-1)/2$ for any
$s \in [1,3] = L_\uM([0,1])$.
Since we know the function $S_\uM$ in
the interval $[1,3]$ we can find, using \eqref{eq:b_B_u_0},
$L_\uM$ in $[1,3]$. We have $\uM(t)=(t+3)/2$ and 
$L_\uM(t)=3(t+1)/2$   
 for $t\in[1,3]$ and so on. 
In general, $S_\uM\colon\oi \to \oi$, 
\[
S_\uM(s)=\frac{k+1}{k+2}(s-T_{k+1})+T_k,\quad
s\in[T_{k+1},T_{k+2}],\quad k\ge0,
\]
and is the inverse of $L_\uM$ and interpolates the values
$ S_\uM(T_k) = T_{k-1}$.

\begin{figure}
\hbox to\hsize{\hfil\includegraphics{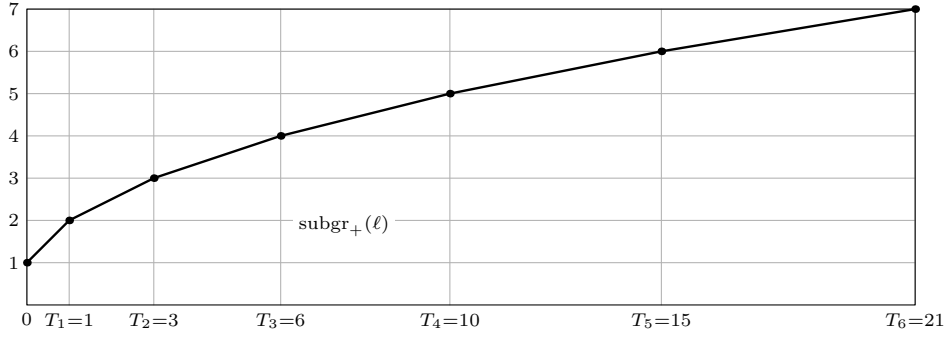}\hfil}
\caption{{\small Graph of
$\uM$ for
Example~\ref{exa:costante_uno}. Time $t$ is horizontal.
}
}
\label{fig:characteristic_half_line}
\end{figure}
\end{Example}

The next interesting example will be completely solved in Section \ref{sec:the_Riemann_problem},
and it is useful in order to understand how to construct $\uM$. 

\begin{Example}[\textbf{Riemann problem in $[0,11+3/2]$}]
\label{exa:the_Riemann_pbm}\rm
Let
\begin{equation}\label{eq:inidat_Riemann_I}
\inidat=1_{[0,1)}.
\end{equation}
Let us 
calculate $\uM
= \uM[\inidat]$,
which we know to be {\it polygonal},
since $\inidat$ is polygonal
(Remark \ref{rem:geometric_meaning}). We shall construct $\uM$ 
 in the 
time interval $[0,11+3/2]$.  Taking $\alpha=1$,
Lemma~\ref{lem:transl} tells us that
\begin{equation}\label{eq:elle_fino_a_uno}
\ell(t) = \inidat(t)+t = 1+t, \qquad t\in(0,1).
\end{equation}
It is convenient to add the vertical segment $\{0\}
\times [0,1] = \{0\}\times [0, \inidat(0)]$
in the generalized graph of $\uM$.
In $[0,1) \times \oi$ the set  ${\rm subgraph}_+(\uM)$
is enclosed between the segment $\{(t, 1+t): t \in [0,1)]\}$
and  the horizontal axis and, by
Proposition
\ref{prop:bound_of_l_from_below},
\begin{equation}\label{eq:striscia_orizzontale}
{\rm subgraph}_+(\uM) \supseteq (0, +\infty) \times (0,1).
\end{equation}
Hence, for $t \in [1,2]$,
the intersection
$I_t = \halfline_t\cap {\rm subgraph}_+(\uM)$
of the half-line $\halfline_t$ with
${\rm subgraph}_+(\uM)$
 is the
segment joining $(t,0)$
with
$(t/2-1/2,t/2+1/2)$, and so
(Remark \ref{rem:geometric_meaning})
\begin{equation}\label{eq:elle_fino_a_due}
\ell(t)=\frac{t}{2}+\frac{1}{2},
\qquad \forall t\in [1,2].
\end{equation}
Therefore the left limit of 
$\uM$ at $1$ is $2 > \uM(1)=1$, and $\uM$ has a decreasing
 jump at $t=1$; it is convenient
to draw the generalized graph of $\uM$ in $[0,2]$, in particular the vertical
segment $\{1\} \times [1,2]$, see
Fig. \ref{fig:riemann-ell}\footnote{
We have furthermore
$\inverseMau_\uM=0$ on $[0,1]$, $\inverseMau_\uM(x) = (x-1)/2$
for $x \in [1,2)$, $\inverseMau_\uM(x) = (x-1)/2+2x/3-4/3=
5x/6 - 11/6$
 for $x \in [2, 7/2]$.}.

Now, using \eqref{eq:elle_fino_a_uno}, \eqref{eq:elle_fino_a_due} and
\eqref{eq:striscia_orizzontale}, we see that
for $t\in[2,3]$ the intersection $I_t
= \halfline_t \cap {\rm subgraph}_+(\uM)$ is {\it not connected}, and splits into
two segments:
one joining $(t,0)$
with
$(t/2-1/2,t/2+1/2)$, and another one
joining
$(1,t-1)$ with
$((2/3)t-1/3,t/3+1/3)$.
Hence
$$
\ell(t)=\frac{t}{2}+\frac{1}{2} + \frac{t}{3} + \frac{1}{3}
-(t-1)
=
\frac{-t+11}{6}
 \qquad \forall t\in[2,3].
$$
For times $t\in [3, 11)$ the intersection $I_t$ turns out
to be connected,
and $\uM$ can be found by iteratively moving on the
right (recall the affine transformation in
Section \ref{subsec:transformation_R_and_rearrangement})  the
point $(2,3/2)$ by $3/2$ (its height), and the point
$(3,4/3)$ by $4/3$ (its height) till the transformed points
have the same first coordinate, i.e., when
$2+(3/2)n=3+(4/3)n$ for $n\in \mathbb N$, giving
$n=6$ and so $t=2 + (3/2)6 =11$.
The graph of $\uM$ in $[3, 11)$ is plotted in Fig.
\ref{fig:riemann-ell}.

Time $t=11$ is similar to $t=1$, when $\uM$ has a decreasing jump
discontinuity, and time $37/3 = 11 + 4/3$ is critical.
For $t\in (11,37/3)$ we find, by a direct computation
similar to the one made in $[2,3]$,
$\ell(t)=(1+t)/9$, while
for $t\in (37/3, 25/2=11+3/2)$ the intersection $I_t$ consists
of two segments: a direct computation gives
$\ell(t)=(1+t)/9 -(t-11)+(1+t)/10=-(71/90)t+1009/90$.
Now since the solution of $37/3+(40/27) n= 25/2+(27/20)n$ is $n=1 =
\lfloor90/71\rfloor$, we have
$t=1009/71$.

The complete recursive construction of $\uM$ is complicated,
and will be done in Section \ref{sec:the_Riemann_problem}.
\end{Example}

\subsection{Initial conditions vanishing at the origin}
\label{subsec:initial_condition_vanishing_at_the_origin}

If we drop assumption \eqref{eq:inidat_larger_than_alpha} we have 
the following result.

\begin{Proposition}
\label{teo:if_we_drop}
Let $\inidat$ be bounded and satisfying 
 \eqref{eq:inidat_from_now_on}.
Then there exists a function $\uM$ 
satisfying \eqref{eq:l_minus_id_is_decreasing} and
\begin{equation}\label{eq:geq}
\uM(t) = \inidat(t) + \vert s \in [0,t] : \uM(t-s) 
\geq s \vert \qquad \forall t \in \oi.
\end{equation}
\end{Proposition}
\begin{proof}
For any $\delta>0$ let $\unknownMau_\delta$ be the unique solution of 
\eqref{eq:how_to_reconstruct_l_from_itself}
 given by
Theorem \ref{teo:existence_of_l},
with $\inidat + \delta$ in place of $\inidat$.
Recall from Lemma \ref{lem:monotonicity} 
that, if $0 < \delta_1 \leq \delta_2$, then 
$\unknownMau_{\delta_1} \leq
\unknownMau_{\delta_2}$. Fix a decreasing sequence $(\delta_n) \subset
(0,+\infty)$ converging
to zero; by monotonicity and since each $\uM_{\delta_n}$ is positive,
$$
\exists \lim_{n \to +\infty}
\uM_{\delta_n}(t) =: \uM(t) \in [0,+\infty)
\qquad \forall t \in \oi.
$$
We know from Theorem \ref{teo:existence_of_l} that 
\begin{equation}\label{eq:gen}
\uM_{\delta_n}(t) 
= \inidat(t) + \delta_n + 
\vert 
\{
\incrt \in [0,t] : 
\uM_{\delta_n}(t-\incrt) 
> 
\incrt 
\}
\vert
\qquad \forall t \in [0,+\infty), \forall n \in \mathbb N,
\end{equation}
so it is sufficient to prove that 
\begin{equation}\label{eq:dominated}
\lim_{n \to +\infty} 
\vert E_n^t\vert = 
\vert E^t\vert
\qquad \forall t \in \oi,
\end{equation}
where 
$$
E_n^t := \{
\incrt \in [0,t] : 
\uM_{\delta_n}(t-\incrt) 
>
\incrt 
\} \supseteq
E^t := \{
\incrt \in [0,t] : 
\uM(t-\incrt) 
\geq
\incrt 
\}.
$$
But
$E^t = \cap_{n \in \NN} E_n^t$,  and so 
\eqref{eq:geq} follows passing to the limit in \eqref{eq:gen}
as $n\to +\infty$.
Also, \eqref{eq:l_minus_id_is_decreasing} follows passing to the
limit, as $n \to +\infty$, in equation \eqref{eq:l_minus_id_is_decreasing}
written with $\unknownMau_{\delta_n}$ in place of $\unknownMau$
and $\inidat + \delta_n$ in place of $\inidat$.
\end{proof}

The next result in particular ensures that if $\inidat$ is zero on 
a right  interval  of $t=0$ and then becomes strictly positive, 
and if $\unknownMau$ exists, then 
also $\unknownMau$ must vanish in that interval; this allows
to use a restarting procedure for 
computing $\unknownMau$, via Theorem \ref{teo:existence_of_l}.

\begin{Proposition} 
Let $\kappa>0$, $a > 0$, and let $\gamma$ be the positive 
root of $\gamma^2 + \kappa \gamma - \kappa = 0$. 
If
\begin{equation}\label{eq:kt}
\inidat(t) \leq \kappa t \qquad \forall t \in [0,a],
\end{equation}
then a solution to \eqref{eq:how_to_reconstruct_l_from_itself}
satisfies
\begin{equation}\label{eq:kt_bis}
\unknownMau(t) \leq \kappa t + \gamma t \qquad \forall t \in [0,a].
\end{equation}
Conversely, 
if 
\begin{equation}\label{eq:kt2}
\inidat(t) \geq \kappa
t \qquad \forall t \in [0,a], 
\end{equation}
then
\begin{equation}\label{eq:kt2_bis}
\unknownMau(t) \geq \kappa t + \gamma t \qquad \forall t \in [0,a].
\end{equation}
\end{Proposition}

\begin{proof} Suppose 
\eqref{eq:kt}. {}From \eqref{eq:how_to_reconstruct_l_from_itself}
 it follows $\unknownMau(t) \leq \kappa
t + t$ for 
any $t \in [0,a]$. Hence 
$$
\{\incrt \in [0,t]: \unknownMau(t-\incrt) 
> \incrt 
\} 
\subseteq
\{\incrt \in [0,t]: 
\incrt < \kappa(t-\incrt) + t-\incrt \} \qquad \forall t \in [0,a].
$$
The inequality $\incrt < \kappa(t-\incrt) + t-\incrt$ implies $\incrt < 
\frac{\kappa+1}{\kappa+2} t =: \gamma_1 t$,
so from \eqref{eq:how_to_reconstruct_l_from_itself} we have the improved estimate
$$
\unknownMau(t) \leq \kappa
 t + \gamma_1 t \qquad \forall t \in [0,a].
$$
We now iterate the argument, and obtain that 
$$
\unknownMau(t) \leq \kappa
 t + \gamma_n t \qquad \forall t \in [0,a],~ \forall n \in \mathbb N^*,
$$
with
$$
\gamma_n := 
\frac{\kappa+\gamma_{n-1}}{1+\kappa + \gamma_{n-1}}, \qquad n \geq 2.
$$
The decreasing 
sequence $(\gamma_n)$ converges to 
the positive solution $\gamma= \frac{-\kappa
 + \sqrt{\kappa^2 + 4\kappa}}{2}$ of the
equation $\gamma = \frac{\kappa+\gamma}{1+\kappa+\gamma}$ or equivalently
$\gamma^2 + \kappa \gamma - \kappa = 0$, and \eqref{eq:kt_bis} follows.

Conversely, suppose \eqref{eq:kt2}. Then, from
\eqref{eq:how_to_reconstruct_l_from_itself} it follows $\unknownMau(t) \geq \kappa t$ for 
any $t \in [0,a]$. Hence 
$$
\{\incrt \in [0,t]: \unknownMau(t-\incrt) 
> \incrt \} 
\supseteq
\{\incrt \in [0,t]: 
\incrt < \kappa(t-\incrt) \}.
$$
The inequality $\incrt  < \kappa(t-\incrt)$ implies $\incrt <
\frac{\kappa}{\kappa+1} t =: \alpha_1 t$,
so that from \eqref{eq:how_to_reconstruct_l_from_itself}
we have the improved estimate
$$
\unknownMau(t) > \kappa t + \alpha_1 t \qquad \forall t \in [0,a].
$$
We now iterate the argument, and obtain that
$$
\unknownMau(t) > \kappa t + \alpha_n 
t \qquad \forall t \in [0,a], \forall n \in \mathbb N^*,
$$
with
$$
\alpha_n := 
\frac{\kappa+\alpha_{n-1}}{1+\kappa + \alpha_{n-1}}, \qquad n \geq 2.
$$
The increasing sequence $(\alpha_n)$ 
converges to $\gamma$, and \eqref{eq:kt2_bis} follows.
\end{proof}


\section{Construction of a solution}\label{sec:construction_of_an_integral_solution}
In this section we prove existence and uniqueness of an
integral solution (Definition \ref{def:integral_solution}).
The next result shows that the function $\unknownMau$ 
studied in Section \ref{sec:another_formulation}
captures
the relevant information to solve
\eqref{eq:pde_intro}.

\begin{Theorem}[\textbf{Existence and uniqueness
of an integral solution}]\label{teo:reconstructing_u_hat_from_l}
Let $\inidat$ satisfy 
\eqref{eq:inidat_from_now_on}
and \eqref{eq:inidat_larger_than_alpha}. 
Let $\uM = \uM[\inidat]:\oi\to \oi$ be the solution to 
\eqref{eq:how_to_reconstruct_l_from_itself}.
Then the 
function $\pressol: \soct \to \oi$ defined as 
\begin{equation}\label{eq:recovering_of_u_hat} 
\pressol(\tau,\xi) 
:= 
\inidat(\xi) + 
\vert \{ s \in [0,\tau] : \uM(s) + s
> \xi\}\vert
\qquad
\forall
(\tau,\xi) \in \soct,
\end{equation} 
is the unique integral solution of 
\eqref{eq:pde_intro}, and 
\begin{equation}\label{eq:l_equal_l_u_hat} 
\pressol(\tau, \tau) = 
\unknownMau(\tau) \qquad \forall \tau \in \oi. \end{equation} 
Moreover:
\begin{itemize} 
\item[(i)] for any $\xi \in \oi$ the function $\pressol(\cdot, \xi)$ is 
one-Lipschitz;
\item[(ii)] $v(\tau,\cdot) \in \BV_{{\rm loc}}([\tau,+\infty))$ for any $\tau >0$;
\item[(iii)] if $\inidat_i$, $i=1,2$, satisfies
\eqref{eq:inidat_from_now_on} and \eqref{eq:inidat_larger_than_alpha}, 
and 
if $\pressol_i$ stands for the expression on the right-hand side of 
\eqref{eq:recovering_of_u_hat} with $\inidat$ replaced by 
$\inidat_i$, 
then 
\begin{equation}\label{eq:comparison}
 \inidat_1 \leq \inidat_2 \Rightarrow \pressol_1 \leq \pressol_2; 
\end{equation} 
\item[(iv)] 
if there 
exists $C>0$ such that for all $\xi \in \oi$ 
\begin{equation}
\label{eq:inidat_Lip_da_una_parte} 
\frac{\inidat(\xi+h)-\inidat(\xi)}{h} \leq C, 
\qquad h >0,
\end{equation} 
then for all $(\tau, \xi),  (\tau, \xi + h) \in \soct$,
\begin{equation}\label{eq:exists} 
\frac{v(\tau,\xi+h)-v(\tau,\xi)}{h} 
\leq C, \qquad h >0;
\end{equation} 
\item[(v)] if 
$\inidat \in L^1((0,+\infty))$ is bounded
then $\pressol(\tau,\cdot) \in L^1((\tau,+\infty))$
and $\pressol$ satisfies the conservation 
of mass\footnote{Note that
$\pressol(\cdot,\xi)$ is increasing;  this is not in contradiction
with the conservation of mass \eqref{eq:conservation_of_mass_tauxi}, 
which is required  in the time-decreasing interval $(\tau,+\infty)$.}
\begin{equation}\label{eq:conservation_of_mass_tauxi}
\int_\tau^{+\infty} \pressol(\tau, \xi) d\xi = \int_0^{+\infty}
\inidat(\xi)~d\xi \qquad
\forall \tau \in \oi.
\end{equation}
\end{itemize} 
\end{Theorem} 
A comment on the  expression of $\pressol$ 
in \eqref{eq:recovering_of_u_hat} is in order. Fix $\overline \xi >0$; 
as usual in linear transport equations, 
we look at the characteristic line $\{\xi = \overline \xi\}$, and 
for a solution  
we should take the value $\inidat(\overline \xi)$ for $(\tau,\xi)$  in the tail region,
and $\inidat(\overline \xi) + \tau$ for $(\tau,\xi)$ in the active region; if the (generalized)
graph $G_{L_\uM}$ of $L_\uM = {\rm id} + \uM$ is a graph with respect to the vertical
$\xi$-axis,
$\pressol(\tau,\xi) = \inidat(\xi)$ in the tail region, while the 
second addendum on the right-hand side of \eqref{eq:recovering_of_u_hat}
activates when $(\tau,\xi)$ is in the active region. 
However, in general $G_{L_\uM}$ needs not be
a graph with respect to the $\xi$-axis (see for instance Fig. \ref{fig:riemann-ell}); 
in case that a characteristic line intersects $G_{L_\uM}$ 
(once or) more then once, we have to add to $\inidat(\overline \xi)$,
both in the tail region and in the active region, 
the Lebesgue measure of the intersection 
of $\{\xi = \overline \xi\}$ with the subgraph of 
$\unknownplustMau_\uM$, i.e.,
the time spent by the characteristic line in the active region.

\begin{proof}
Nonnegativity of $v$ is immediate, 
since $\inidat$ is nonnegative.

\noindent (i) $v$ is one-Lipschitz in $\tau$, 
since clearly, for any $\xi \geq 0$ and any $h \in \R$ with $\xi+
h \geq 0$,
$$
\begin{aligned}
\vert
v(\tau + h, \xi) - v(\tau,\xi)
\vert =&
\Big\vert \vert 
\{ s \in [0,\tau+h] : \uM(s) +s > \xi\} \vert
\\
& -
\vert \{ s \in [0,\tau] : \uM(s) +s> \xi\} 
\vert 
\Big\vert \leq \vert h\vert.
\end{aligned}
$$
\noindent (ii) $v(\tau, \cdot) \in \BV_{{\rm loc}}([\tau, +\infty))$, since $\inidat
\in \BV_{{\rm loc}}(\oi)$ and the second addendum on the right-hand
side of \eqref{eq:recovering_of_u_hat} is nonincreasing (and right-continuous) if
considered as a function of $\xi$. 

Let $\tau >0$; passing to the limit as $\xi \to\tau^+$ in  \eqref{eq:recovering_of_u_hat} and using the right-continuity
of the right-hand side, gives
$$
\lim_{\xi\to \tau^+} \pressol(\tau, \xi)= \pressol(\tau, \tau)
=
\\
\inidat(\tau)  + \vert \{ s \in [0,\tau] : 
\uM(s)+s > \tau\}\vert = \uM(\tau),
$$
where the last equality follows from
\eqref{eq:how_to_reconstruct_l_from_itself}, 
and \eqref{eq:l_equal_l_u_hat} follows.

{}From \eqref{eq:recovering_of_u_hat} and \eqref{eq:l_equal_l_u_hat} it follows
$$
\pressol(\tau,\xi) =
\inidat(\xi) + \vert \{s \in [0,\tau] : 
v(s,s)+s > \xi\}\vert,
$$
i.e., we have that
\eqref{eq:integral_is_equivalent} holds, which is equivalent to
\eqref{eq:equazione_integrale_0}.

Concerning uniqueness, suppose that 
$w$ is another integral solution of \eqref{eq:pde_intro}; from \eqref{eq:integral_is_equivalent}
it follows
$$
w(\tau, \tau)
=
\inidat(\tau)  + \vert \{ s \in [0,\tau] : 
w(s,s)+s > \tau\}\vert \qquad \forall \tau \in [0,+\infty).
$$
Thus $w(\tau, \tau)$ satisfies the same equation as $\uM(\tau)$;
from the uniqueness property stated in Theorem \ref{teo:existence_of_l}, 
we get $w(\tau,\tau)=  \uM(\tau)$ for any $\tau \in \oi$. Thus
$$
v(\tau, \tau) = w(\tau, \tau) \qquad \forall \tau \in \oi,
$$
and so $v=w$ from \eqref{eq:recovering_of_u_hat} and
\eqref{eq:l_equal_l_u_hat}.

\noindent (iii) If $\inidat_1 \leq \inidat_2$, 
then \eqref{eq:comparison} follows straightforwardly
 from \eqref{eq:recovering_of_u_hat}
and \eqref{eq:how_to_reconstruct_l_from_itself}.

\noindent (iv) It immediately follows from \eqref{eq:inidat_Lip_da_una_parte},
\eqref{eq:recovering_of_u_hat}, 
and the inclusion $\{
s \in [0,\tau] : s + \unknownMau(s) >\xi + h\} 
\subseteq \{s \in [0,\tau] : s + \unknownMau(s) > \xi\}$.

\noindent (v) 
Given $\tau >0$, $\pressol(\tau, \cdot) \in L^1((\tau, +\infty))$,
since $\uM$ is bounded on $[0,T]$ for any $T>0$, and so the integral
$\int_\tau^{+\infty} \vert \{ s \in [0,\tau]: 
\uM(s) + s > \xi \}\vert~d\xi$ actually reduces to an integral
on a bounded interval. 
Using the boundedness of $\inidat$ and 
that $\pressol(\cdot,\xi)$ is Lipschitz, it follows
that 
$\int_{\tau}^{+\infty} \pressol(\tau,\xi)~d\xi$ is locally 
absolutely continuous
in $\tau$, hence differentiable on a set $I \subseteq (\tau,+\infty)$ of
full measure. Since $\uM\in \BV_{{\rm loc}}(\oi)$, almost every point of $I$ is a continuity point of 
$\uM$. Fix such a point $\tau$:  writing the limits
of the incremental quotients around $\tau$ we get
\begin{equation*}
\begin{aligned}
\frac{d}{d\tau} \int_{\tau}^{+\infty} \pressol(\tau,\xi)~d\xi =& 
- \pressol(\tau, \tau) + 
\int_{\tau}^{+\infty} 
\pressol_\tau(\tau,\xi)
~d\xi 
\\
=&
- \pressol(\tau, \tau) + 
\int_{\tau}^{+\infty} 
1_{(\tau, \pressol(\tau,\tau) + \tau)}(\xi) 
~d\xi 
\\
=&
- \pressol(\tau, \tau) + 
\tau + v(\tau,\tau) -  \tau =0.
\end{aligned} 
\end{equation*}
\end{proof}

The next observation can be used to quickly deduce $v$
from the knowledge of $\uM$, and it is used to find
the solution of the Riemann problem,
see the pictures in Figs. \ref{fig:Riemann_movie_I}, \ref{fig:Riemann_movie_II}
and also Remark
\ref{rem:from_l_to_v}. 

\begin{Remark}[\textbf{Finding $v$ from $\uM$}]\rm\label{rem:alternative_expression}
If $\pressol$ is as in \eqref{eq:recovering_of_u_hat},
it is immediate to check that 
\begin{equation}\label{eq:part_imp}
\begin{aligned}
\pressol(\tau,\xi)=&
\unknownMau(\xi) - 
|\{s\in[\tau,\xi] : 
\unknownMau(s) + s> \xi \}| 
\\
= & 
\unknownMau(\xi) - (\xi-\tau) + 
|\{s\in[\tau,\xi] : 
\unknownMau(s) \leq \xi-s \}| 
\end{aligned}
\qquad  \forall(\tau,\xi)\in \soct.
\end{equation}
Indeed, from 
\eqref{eq:how_to_reconstruct_l_from_itself}
and \eqref{eq:obv}, 
$$
\begin{aligned}
& \unknownMau(\xi)  = \inidat(\xi) + 
|\{s\in[0,\xi] : 
\unknownMau(s) + s> \xi \}| 
\\
= &
\inidat(\xi) + 
|\{s\in[0,\xi] : 
\unknownMau(s) + s>  \xi \}| 
\\
& 
+(\xi-\tau) 
- 
|\{s\in[\tau,\xi] : 
\unknownMau(s) + s\leq \xi \}| 
- 
|\{s\in[\tau,\xi] : 
\unknownMau(s) + s> \xi \}| 
\\
= &
\inidat(\xi) + 
|\{s\in[0,\tau] : 
\unknownMau(s) + s>  \xi \}| 
+(\xi-\tau) 
- 
|\{s\in[\tau,\xi] : 
\unknownMau(s) + s\leq \xi \}|  
\\
= &
v(\tau,\xi) +
|\{s\in[\tau,\xi] : 
\unknownMau(s) + s> \xi \}|,
\end{aligned}
$$
where the last equality follows from \eqref{eq:recovering_of_u_hat}.
\end{Remark}

Without passing to coordinates $(\tau,\xi)$ a possible
definition of solution (less transparent than Definition 
\ref{def:integral_solution})
is the following. 

\begin{Definition}[\textbf{Distributional solution}]
\label{def:distributional_solution}
We say that a nonnegative function
$v \in L^1_{{\rm loc}}( \firstquadrantopen)$
defined everywhere is a distributional solution to \eqref{eq:limit_problem_x} if
 for any $t> 0$ there exists $\lim_{x \to 0^+}
u(t,x) =: u(t,0) \in \oi$, and 
for any $\varphi \in \mathcal C^1_c(\oi \times (0,+\infty))$ we have
\begin{equation}\label{eq:distributional_solution}
\begin{aligned}
& \int_0^{+\infty}
\int_0^{+\infty}
-u (\varphi_t - \varphi_x)
~dxdt
-
\int_0^{+\infty}
\int_0^{+\infty}
1_{(0,u(t,0))}(x)\varphi(t,x)
~dxdt
\\
& + \int_0^{+\infty} \inidat(x)
\varphi(0,x)~dx=0.
\end{aligned}
\end{equation}
\end{Definition}

\begin{Proposition}[\textbf{Integral solutions and distributional solutions}]
Let $\wu = \wu(\tau,\xi)$ be an integral solution of
\eqref{eq:limit_problem_x} as in Theorem 
\ref{teo:reconstructing_u_hat_from_l}. Then $u := \wu(\Phi)
=u(t,x)$ is a distributional solution of \eqref{eq:limit_problem_x}.
Conversely,
let $u = u(t,x)\in \BV_{{\rm loc}}(\firstquadrantopen)$ be a locally bounded
distributional solution of
\eqref{eq:limit_problem_x}.  
Then $\wu:= u(\Phi^{-1})$ satisfies (i) of Definition \ref{def:integral_solution}
and (ii) for any $\tau > 0$ and for almost every $\xi$ with $(\tau, \xi) \in \soctopen$.
\end{Proposition}
\begin{proof}
Let $\varphi = \varphi(t,x) \in \mathcal C^1_c(\oi \times (0,+\infty))$, and
set $\widehat \varphi :=\varphi(\Phi^{-1})$.

Suppose that $\wu$ is an integral solution.
{}From \eqref{eq:recovering_of_u_hat} 
we have
$$
u(t,x) 
= 
\inidat(t+x) + 
\vert \{ s \in [0,t] : \uM(s) + s > x+t\}\vert,
\qquad (t,x) \in \firstquadrantopen.
$$
%
Hence

\[
\begin{aligned}
& -\int_{\firstquadrantopen} u(\varphi_t-\varphi_x)~dtdx=
-\int_{\firstquadrantopen} \inidat(t+x)(\varphi_t(t,x)-\varphi_x(t,x))\, dtdx
\\
& -\int_{\firstquadrantopen}\vert \{s \in [0,t] : \unknownMau(s) +s > x+t \} \vert
\bigl(\varphi_t(t,x)-\varphi_x(t,x)\bigr)\, dtdx\\
=& -\int_0^{+\infty} \inidat(\xi)
\biggl(\int_0^\xi\widehat\varphi_\tau(\tau,\xi)\, d\tau\biggr)\, d\xi\\
& -\int_0^{+\infty}
\biggl(\int_0^\xi \vert 
\{s \in [0,\tau] : \unknownMau(s) + s > \xi \} 
\vert
\widehat\varphi_\tau(\tau,\xi)\, d\tau\biggr)\, d\xi
=:\mathrm{I}+\mathrm{II}.
\end{aligned}\]
It is immediate that
\[
\mathrm{I} = 
\int_0^{+\infty} 
\widehat \varphi(0,\xi) \inidat(\xi)\,d\xi = 
\int_0^{+\infty} \varphi(0,x) \inidat(x)\,dx.
\]
Moreover, using \eqref{eq:so_that}, 
\[
\vert \{s \in [0,\tau] : \unknownMau(s) + s > \xi \} \vert
=\int_0^\tau 1_{(s,\unknownMau(s)+s)}(\xi)\, ds\qquad \forall
\xi\ge\tau,
\]
and hence
\[
\begin{aligned}
\mathrm{II}&=-\int_0^{+\infty}
\biggl(\int_0^\xi\Bigl(\int_0^\tau 1_{(s,\unknownMau(s) + s)}(\xi)\, ds
\Bigr)
\widehat\varphi_\tau(\tau,\xi)\, d\tau\biggr)\, d\xi\\
&=\int_0^{+\infty}
\biggl(\int_0^\xi 1_{(\tau,\unknownMau(\tau) + \tau)}(\xi)
\widehat\varphi(\tau,\xi)\, d\tau\biggr)\, d\xi=
\int_{\firstquadrantopen}1_{(0,u(t,0))}(x)
\varphi(t,x)\, dtdx,
\end{aligned}
\]
where the second equality follows integrating by parts,
and the last equality follows from 
\eqref{eq:l_equal_l_u_hat}. 
This proves that $u$ is a distributional solution of 
\eqref{eq:limit_problem_x}.

\smallskip

Conversely, suppose that 
$u \in \BV_{{\rm loc}}(\firstquadrantopen) \cap L^\infty_{{\rm loc}}(\firstquadrantopen)$ is 
a distributional solution of \eqref{eq:limit_problem_x}.  
{}From \eqref{eq:distributional_solution} 
it follows
\[
\begin{aligned}
& -\int_{0}^{+\infty}
\int_0^\xi\wu(\tau,\xi)
\widehat\varphi_\tau(\tau,\xi)\, d\tau d\xi-
\int_0^{+\infty}
\int_0^\xi 
1_{(\tau,\wu(\tau,\tau)+\tau)}(\xi)\widehat \varphi(\tau,\xi) ~d\tau d\xi
\\
& +\int_0^{+\infty} \inidat(\xi) \widehat \varphi(0,\xi)~d\xi=0.
\end{aligned}
\]

Taking $\varphi$ with compact support in 
$\soctopen$, it follows that the distributional
partial derivative $D_\tau \wu$, which is a measure with
locally finite total variation in $\soctopen$, satisfies 
\begin{equation}\label{eq:measures}
D_\tau \wu(\tau,\xi)
=  1_{(\tau,\wu(\tau,\tau) + \tau)}(\xi)
\end{equation}
in the sense of measures
in 
$\soctopen$.
Since the right-hand side of
\eqref{eq:measures} is a locally integrable function (taking 
values in $\{0,1\}$), 
$D_\tau \wu$ coincides with its absolutely
continuous part, the density of which we denote by 
$\nabla_\tau \wu$. 
For any $\xi \in \oi$ consider the slice $\widehat u^{(\xi)} : [0,\xi]\to \R$
of $\wu$, i.e., the restriction of $u$ in $\firstquadrant$ to the horizontal line
passing through $\xi$. 
By \cite[Proposition 4.35]{AmFuPa:00}
for almost every $\xi 
\in \oi$ 
we have 
$\wu^{(\xi)} \in \BV_{{\rm loc}}([0,\xi])$, its 
distributional derivative 
$\dot {\wu}^{(\xi)}$ 
is absolutely continuous,
and $\dot {\wu}^{(\xi)}(\tau) = \nabla_\tau \wu(\tau,\xi)$ for almost
every $\tau \in [0, \xi]$. We deduce
$$
\dot {\wu}^{(\xi)}(\tau) 
=  1_{(\tau,\wu(\tau,\tau)+ \tau)}(\xi)
\qquad {\rm for~a.e.~} \tau \in [0,\xi].
$$
Integrating, \eqref{eq:equazione_integrale_0} follows for almost any $\xi \geq 
0$ and for almost every $\tau \in [0,\xi]$, and hence for any $\tau
\in [0,\xi]$.
\end{proof}

\section{A Lyapunov functional}
\label{sec:a_Lyapunov_functional}
The  functional 
$\Lyap : D(\Lyap) \subset L^2(\R) \to [0,+\infty)$, 
corresponding to the continuous version of 
\eqref{eq:discrete_Lyapunov_functional},
is

$$
\Lyap(u) := \int_0^{+\infty} \int_x^{x + u(x)} y ~dy ~dx,
$$
where the domain $D(\Lyap)$ of $\Lyap$ consists
of all {\it nonnegative}
$u \in L^2((0,+\infty))$ such that $xu \in L^1((0,+\infty))$, and
the height $y$ has the meaning of a potential energy.
It can be evaluated as
$$
\Lyap(u)
 = \frac{1}{2} \int_0^{+\infty} \left( u^2(x) + 2xu(x)\right)~dx
     = \frac{1}{2} \|u\|^2_{L^2} + (x, u)_{L^2}
$$
and is a strictly convex functional.

\begin{Proposition}[\textbf{Stationary solutions}]
        Fix $\mass > 0$. Then the solution of the variational problem
        $$
\min\Big\{ \Lyap(u): u \in D(\Lyap), 
      \int^{+\infty}_0 u(x)~dx= \mass
\Big\},
$$
is
        $$u(x) = \max\left\{\sqrt{2\mass}-x,0\right\} \qquad \forall x \in \oi.$$
\end{Proposition}
\begin{proof}
       Given
$u \in D(\Lyap)$, introduce 
the function $\varphi = \sqrt{u}$, so that 
        $$u(x) = \varphi^2(x)  \quad \text{ for every } \quad x \geq 0.$$
        In this way the pointwise constraint $u(x) \geq 0$ is trivially satisfied for every $x \geq 0$ and the variational problem
reads as follows:
        $$\min_{\varphi} \quad \frac{1}{2}
\int_0^{+\infty} \left( \varphi^4(x) + 2x\varphi^2(x)\right)~dx \quad \text{ subject to } \int^{+\infty}_0 \varphi^2(x)~dx = \mass.
$$
        The Euler-Lagrange equation associated to the above variational problem is
        \begin{align*}
        4\varphi^3(x) + 4x\varphi(x) -2\lambda\varphi(x)
= 0 \iff \varphi(x)\Bigg(\varphi^2(x)+x-\frac{\lambda}{2}\Bigg) = 0,
        \end{align*}
        where $\lambda$ denotes the Lagrange multiplier of the mass constraint. Such a condition is clearly satisfied if
        $$\varphi \equiv 0 \quad \text{ or } \quad 
\varphi^2(x) = \frac{\lambda}{2} - x \quad \forall x \in \oi,$$
        or, equivalently
        $$u \equiv 0 \quad \text{ or } \quad u(x) = \frac{\lambda}{2} - x \quad
\forall x \in \oi.$$
        We have thus determined that a stationary point of the variational problem has the form
        $$u(x) = \max\left\{\frac{\lambda}{2} - x, 0\right\} \qquad \forall 
x \in \oi.$$
        In order to determine $\lambda$ we impose the mass constraint on $u$, obtaining
        \begin{align*}
        \mass
 = \int^{+\infty}_0 u(x) dx = \int^{\lambda/2}_0 \left(\frac{\lambda}{2} - x\right)dx = \Big(\frac{\lambda^2}{4} - \frac{\lambda^2}{8}\Big) = \frac{\lambda^2}{8},
        \end{align*}
        which implies $\lambda = 2\sqrt{2\mass}.$ Plugging it in the expression for $u$
and recalling that $\Lyap$ is strictly convex,  we get the statement.
\end{proof}

\subsection{$\Lyap$ decreases along a solution}
\label{sec:Lyap_decreases_along_a_solution}
Let $\inidat \in D(\Lyap) \cap L^\infty((0,+\infty))$ satisfy 
\eqref{eq:inidat_larger_than_alpha} 
and denote by $\wu$ the solution given by Theorem \ref{teo:reconstructing_u_hat_from_l}. 
In coordinates $(\tau,\xi)$ we have

$$
\Lyap(\wu(\tau))=   \int_{\tau}^{+\infty}\int_{\xi-\tau}^{\xi-\tau+\wu(\tau,\xi)}
y~dy d\xi = 
\frac{1}{2} \int_{\tau}^{+\infty}
\left( \wu(\tau,\xi)^2 + 2 (\xi-\tau) \wu(\tau,\xi)\right)
d\xi,
$$
where $\wu(\tau)(\cdot) := \wu(\tau,\cdot)$, and one checks that $\wu(\tau)
\in D(\Lyap)$ for any $\tau \geq 0$. 
We claim that for all $\sigma,\tau \in \oi$
with $\sigma < \tau$ we have
\begin{equation}\label{eq:dissipation_equality}
\Lyap(\wu(\tau)) + \int_\sigma^\tau \int_{r + \uM(r)}^{+\infty} \wu(r,\xi)~d\xi~dr 
=\Lyap(\wu(\sigma)),
\end{equation}
where we recall that $\uM(r)=\wu(r,r)$, with $\uM$ given by 
Theorem \ref{teo:existence_of_l}. 

Since 
$\wu(\cdot, \xi)$ is Lipschitz, the function 
$\tau \to \Lyap(\wu(\tau))$ is locally absolutely continuous. 
Hence, at each of its differentiability points which
are also continuity points for $\uM$ (hence, at 
almost
every $\tau \geq 0$) we have, using also \eqref{eq:l_equal_l_u_hat},
\begin{equation}\label{eq:derivative_Lyap}
\begin{aligned}
& \frac{d}{d\tau} 
\Lyap(\wu(\tau))
=
-\frac{1}{2}\ell(\tau)^2+\int_\tau^{+\infty}
[\wu(\tau,\xi)+\xi-\tau]\wu_\tau(\tau,\xi)\, d\xi-\int_\tau^{+\infty}
\wu(\tau,\xi)\, d\xi\\
=&
-\frac{1}{2}\ell(\tau)^2+\int_\tau^{+\infty}
[\wu(\tau,\xi)+\xi-\tau]1_{(\tau,\tau+\ell(\tau))}(\xi)\, d\xi
-\int_\tau^{+\infty}
\wu(\tau,\xi)\, d\xi\\
\\
= & 
-\frac{1}{2}\ell(\tau)^2-\int_{\tau + \uM(\tau)}^{+\infty}
\wu(\tau, \xi)~d\xi + 
\int_\tau^{\tau + \uM(\tau)} (\xi-\tau)~d\xi
\\
=&-\int_{\tau+\ell(\tau)}^{+\infty}\wu(\tau,\xi)\, d\xi \leq 0.
\end{aligned}
\end{equation}
The nonnegative (double) integral term in \eqref{eq:dissipation_equality} 
can be considered as a dissipated quantity by the system, 
when passing from time
$\sigma$ to a later time $\tau$.

\section{Examples}\label{sec:examples}
Given $\inidat$, our strategy to construct a solution
of \eqref{eq:pde_intro} is as follows: 
first compute $\uM$ using one of the methods illustrated
in Section \ref{sec:another_formulation} 
(see Remark \ref{rem:summa}); next, using \eqref{eq:recovering_of_u_hat},
compute an integral solution in variables $(\tau, \xi)$ and,
whenever convenient, 
re-express it in variables $(t,x)$. 

\begin{Example}[\textbf{Stationary solutions}]
\label{exa:stationary_solutions}\rm
Let $m \geq 0$, 
\begin{equation}\label{eq:stationary_solution_a}
\inidat(x) := 
\max\Big\{
\sqrt{2\mass}  
-x, 0
\Big\}
 \qquad  \forall x \geq 0,
\end{equation}
so that
 $\mass=\int_{\oi} \inidat~dx$. 
Then 
\begin{equation}\label{eq:stationary_solution_b}
u(t,x) := u_0(x), \qquad (t,x) \in \oi \times \oi
\end{equation}
is a stationary solution of \eqref{eq:pde_intro}, and 
$\uM(t) = \sqrt{2\mass}$ for any $t \geq 0$,
$$
\freebound(u)  = \{(t,\sqrt{2\mass}) : t \geq 0\}.
$$
In  $(\tau,\xi)$-variables, 
$u(\tau,\xi) = \inidat(\xi-\tau)$ is
a travelling wave.
\end{Example}

The next example concerns the initial condition in
Example \ref{exa:costante_uno}: it
shows that maxima can increase, and also
that smoothness can be lost.

\begin{figure}
\hbox to\hsize{\hfil\includegraphics{\figdir/cman-2.mps}\hfil}
\caption{{\small Graph of
$L_\uM$ for 
Example~\ref{exa:characteristic_of_a_half_line}.
}}
\label{fig:characteristic_half_line_L}
\end{figure}

\begin{Example}[\textbf{Constant initial condition, II}]
\label{exa:characteristic_of_a_half_line}
Let 
\begin{equation}\label{eq:inidat_11}
u_0=1_{[0,+\infty)},
\end{equation}
which has infinite mass.  Recall that $\uM$ and $L_\uM$
have been computed in \eqref{eq:exa_l_1} and 
\eqref{eq:exa_L_1}, see Figs. \ref{fig:characteristic_half_line}
and \ref{fig:characteristic_half_line_L}.

For $(t,x)$ in the interior of $\soctopen \setminus 
{\rm subgraph}_+(\uM)$ we have $u(t,x)=1$. 
Observe that
$$
(t,x)\in 
{\rm subgraph}_+(\uM)
\cap ([T_k,T_{k+1}]\times[0,+\infty)) \Rightarrow
t+x \in [T_k,T_{k+1}] \cup [T_{k+1}, T_{k+2}],
$$
and we have
$t+x\in[T_k,T_{k+1}]$ when $0\le x\le -t+T_{k+1}$ while $ t+x\in[T_{k+1},
T_{k+2}]$ when $-t+T_{k+1}\le x< \uM(t)$.
Equivalently, in $(\tau, \xi)$ variables, we can split 
${\rm subgraph}_+(\uM)\cap ([T_k,T_{k+1}]\times[0,+\infty))$ as the union
of two disjoint regions:
$$
\begin{aligned}
S_{{\rm I}} :=& \{(\tau,\xi)\in \soct : T_k \leq \tau \leq T_{k+1}, 
\tau \leq \xi \leq T_{k+1}\},
\\
S_{{\rm II}} :=& \left\{(\tau,\xi) \in \soct: T_k \leq \tau \leq T_{k+1}, 
T_{k+1} < \xi < L_\uM(\tau) = 
\tau + \frac{\tau}{k+1} + \frac{k}{2}+1
\right\}.
\end{aligned}
$$
For $(\overline \tau, \overline \xi) \in S_{{\rm I}}$,
in order to find the time spent by the characteristic line
$\{\xi = \overline \xi\}$
in ${\rm subgraph}_+(\uM)$ before reaching the vertical axis, 
we need to compute the $\tau$-coordinate
of the intersection point of $\{\xi = \overline \xi\}$
with ${\rm graph}(L_\uM)_{\vert [T_{k-1}, T_k]}$; with 
$\overline \xi = L_\uM(\tau) = 
\frac{k+1}{k}\tau + \frac{k+1}{2}$ (see \eqref{eq:exa_L_1}) we get
$$
\tau = 
\frac{k}{k+1}
\Big(\overline \xi-\frac{k+1}{2}\Big).
$$
The time spent is therefore $\overline \tau  - \tau = 
\overline \tau - \frac{k}{k+1}
(\overline \xi-\frac{k+1}{2})$ and so, recalling \eqref{eq:recovering_of_u_hat}
and \eqref{eq:inidat_11}, 
\begin{equation}\label{eq:time_spent_I}
v(\overline \tau, \overline \xi) = 1 + \overline \tau 
-\frac{k}{k+1}
\Big(\overline \xi-\frac{k+1}{2}\Big) \qquad \forall (\overline \tau, \overline \xi)
\in S_{\rm I}.
\end{equation}
For $(\overline \tau, \overline \xi) \in S_{{\rm II}}$,
we need the $\tau$-coordinate
of the intersection point of $\{\xi = \overline \xi\}$
with ${\rm graph}(L_\uM)_{\vert [T_{k}, T_{k+1}]}$; with 
$\overline \xi = L_\uM(\tau) = 
\frac{k+2}{k+1}\tau + \frac{k+2}{2}$ we get
$$
\tau = 
\frac{k+1}{k+2}
\Big(\overline \xi-\frac{k+2}{2}\Big).
$$
The time spent is therefore $\overline \tau  - \tau = 
\overline \tau - \frac{k+1}{k+2}
\Big(\overline \xi-\frac{k+2}{2}\Big)$, and so 
\begin{equation}\label{eq:time_spent_II}
v(\overline \tau, \overline \xi) = 1 + \overline \tau 
-\frac{k+1}{k+2}
\Big(\overline \xi-\frac{k+2}{2}\Big) \qquad \forall (\overline \tau, \overline \xi)
\in S_{\rm II}.
\end{equation}
Note (see Fig. \ref{fig:exa_semiretta}) that for any $k \in \NN$, 
\begin{itemize}
\item[(i)]
for any $\overline \tau \in (T_k, T_{k+1})$, the derivative
of $v(\overline \tau,\cdot)$ has two  jumps, corresponding
to the intersection of $\{\tau = \overline \tau\}$
with $\{\xi=T_{k+1}\}$ and with 
${\rm graph}(L_\uM)_{\vert (T_k, T_{k+1})}$;
\item[(ii)] for any $\overline \xi \in (T_k, T_{k+1})$, the derivative
of $v(\cdot,\overline \xi)$ has one jump, corresponding 
to the intersection of $\{\xi = \overline \xi\}$
with ${\rm graph}(L_\uM)_{\vert (T_k, T_{k+1})}$.
\end{itemize}
Going back to $(t,x)$-coordinates,
from \eqref{eq:time_spent_I} and  \eqref{eq:time_spent_II} we have, for any $k\ge0$
and any $t\in[T_k,T_{k+1}]$,
\begin{equation}
u(t,x)=\begin{cases}
\frac{t}{k+1} -\frac{k}{k+1} x +1+\frac{k}{2} & \hbox{for $0\le x\le
-t+T_{k+1}$},\\
\frac{t}{k+2} -\frac{k+1}{k+2} x +1+\frac{k+1}{2}
& \hbox{for $-t+T_{k+1}< x\le \frac{t}{k+1}+1+\frac{k}{2}$},\\
1 & \hbox{for $x> \frac{t}{k+1}+1+\frac{k}{2}$}.
\end{cases}
\end{equation}
Function $u$ is Lipschitz, piecewise affine,
it forms an initial plateau, originating from $x=1$,
that moves vertically upwards with 
speed one, which is next linearly interpolated with the constant
one (which does not move, and is eroded),
the $x$-slope of the interpolants being 
equal to $-1/2$. The two points (both 
originating from $x=1$) on the $x$-axis
corresponding to the two corners in the graph 
of $u(t,\cdot)$, $t \in (0,1]$ move with unit speed, one toward
the left and the other one toward the right; 
see Fig. \ref{fig:exa_semiretta}. 
\end{Example}
\def\+#1+{\includegraphics{\figdir/half-#1.mps}}
\begin{figure}
\hbox to\hsize{\+1+\hfil\+2+\hfil\+3+\hfil\+4+}
\bigskip
\hbox to\hsize{\+5+\hfil\+6+\hfil\+7+\hfil\+8+}
\bigskip
\hbox to\hsize{\+9+\hfil\+10+\hfil\+11+\hfil\+12+}
\bigskip
\hbox to\hsize{\+13+\hfil\+14+\hfil\+15+\hfil\+16+}
\bigskip
\hbox to\hsize{\+17+\hfil\+18+\hfil\+19+\hfil\+20+}
\bigskip
\hbox to\hsize{\+21+\hfil\+22+\hfil\+23+\hfil\+24+}
\caption{Time-slices of the solution $u$ 
in Example~\ref{exa:characteristic_of_a_half_line}
\nada{
$k=0$, $(t,x) \in [0,1] \times 
\oi$ and
$u(t,x) = 
t+1$ if 
$0 \leq x \leq -t+1$ and 
$t \in [0,1]$,
$u(t,x) = 
\frac{t}{2} 
-\frac{x}{2} 
+ \frac{3}{2}$ if
$-t+1 \leq x \leq t+1$ and 
$t \in [0,1]$,
$u(t,x)= 1$ if 
$x \geq t+1$ and $t \in [0,1]$.
}
}
\label{fig:exa_semiretta}
\end{figure}

\begin{Example}[\textbf{Increasing jump}]\rm
Let $\inidat(x) =1_{[0,1)} + 2 1_{[1,+\infty)}$. For 
$t \in [0,1]$ a solution is 
$u(t,x) = 1_{[1-t, 1)}(x) + (2+t) 1_{[1,1+t)} + 2 1_{[1+t,+\infty)}$.
In this case, the initial increasing jump travels
toward the left at unit constant speed and then 
disappears, and meanwhile a new decreasing jump
is formed during the evolution. 
\end{Example}



\section{The Riemann problem}\label{sec:the_Riemann_problem}
The next example, which has been
initially discussed in Example \ref{exa:the_Riemann_pbm},
 exhibits interesting phenomena (Figs.
\ref{fig:Riemann_movie_I}, \ref{fig:Riemann_movie_II}), and shows in particular
a solution 
for which:
\begin{itemize}
\item[(i)]  
the initial jump persists for some time, it moves with unit speed
toward the origin, and then disappears;
\item[(ii)] 
at suitable later times new jumps may form (and persist for some time
with a similar behaviour as in (i)); 
\item[(iii)] there are countably many critical 
times, and so vertical rearrangement is necessary
a countable number of times; 
\item[(iv)] the solution converges to a 
stationary solution in infinite time;
\item[(v)] the Lyapunov functional along the solution is Lipschitz
and is 
constant excluding segments which are projection
on the $t$-axis
of critical segments,
where instead it strictly decreases.
\end{itemize}

Let
\begin{equation}\label{eq:inidat_Riemann}
\inidat=1_{[0,1)}. 
\end{equation}

\begin{figure}
\hbox to\hsize{\hfil\includegraphics{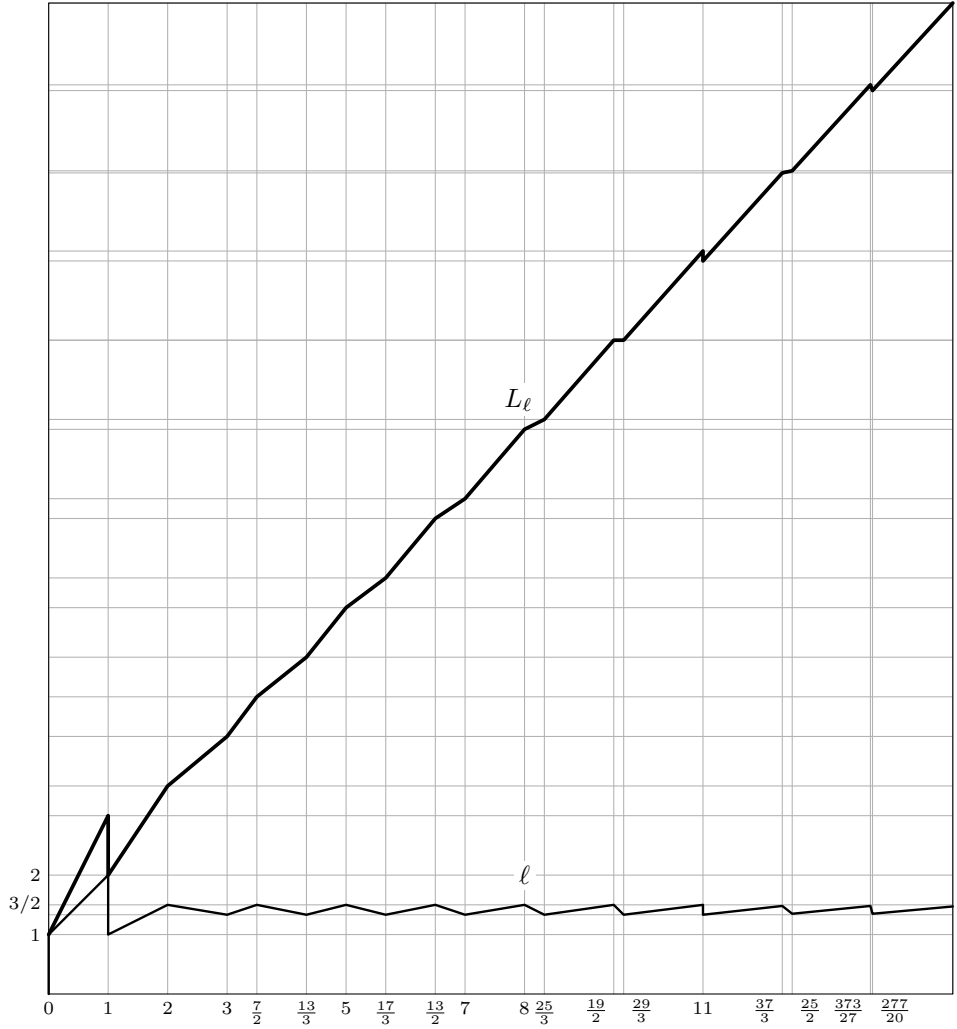}\hfil}
\caption{
{\small 
The Riemann problem with $\inidat$ in \eqref{eq:inidat_Riemann}: 
graphs of $\uM$ (the polygonal defined by \eqref{eq:polP}) and of $L_\uM$ 
(the bold one).
At the local maximum 
$T_9=37/3$ the value of $\uM$ is $\lmax_9=40/27$ which is
 strictly smaller 
than the one $(\lmax_8=3/2)$
at time $T_8=11$. At the
first local minimum $t_9 > T_9$ the value of $\uM$ is 
$\lmin_9=27/20$ 
which is strictly  larger than the value $(\lmin_8=4/3)$ of $\uM$ at $\tmin_8$.
The first
critical time (Definition \ref{def:critical_segments_and_critical_times})  is $t=1$, and the second one is $t = 37/3$.
The distance $\tmin_{n+1} - \tmin_n = 4/3$
between two local minimizers (after $t=1$) 
remains constant before the first critical time; next it slightly 
increases,
remaining constant before the second critical time, and so on.} 
}
\label{fig:riemann-ell}
\end{figure}

In Example \ref{exa:the_Riemann_pbm} we have found $\uM$ in $[0, 11+3/2]$; here
we give the general rule for finding $\uM$ (recall that
$\uM$ may have jumps, and over a jump point
there is a vertical segment in its generalized graph; correspondingly,
we have a minimal and a maximal value of $\uM$;
for convenience,
the generalized graph of $\uM$ contains the initial vertical segment $\{0\} \times [0,1]$).

\begin{figure}
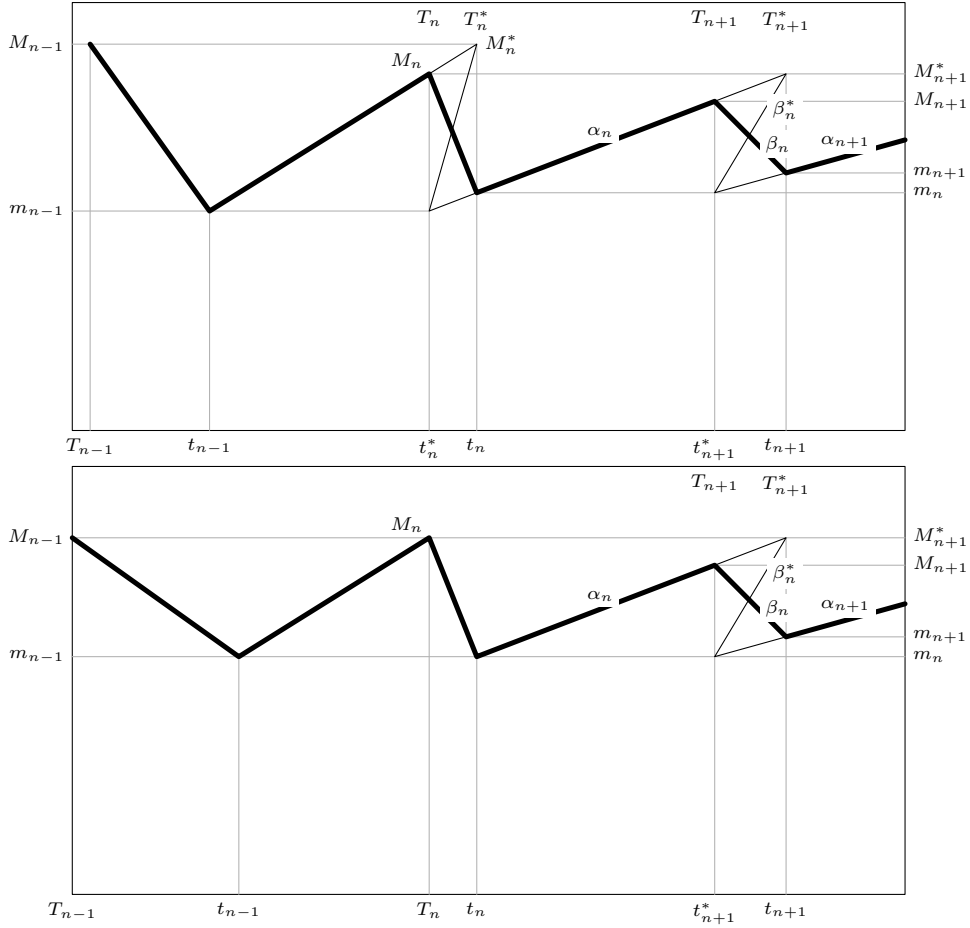

\begin{center}
\includegraphics{\figdir/slopes-2.mps}
\includegraphics{\figdir/slopes-1.mps}
\end{center}
\caption{\small Vertical rearrangement takes place in the 
interval $(\tmins_{n+1},
\tmaxs_{n+1})$ where the thin zig-zag line 
(which is part of the polygonal $\mathcal {P}^*$)
is no-longer a graph.
The result of the rearrangement is the bold decreasing
segment (which is part of the polygonal $\mathcal P$) with negative slope $\beta_n$ and can be obtained by linear
interpolation.
In the previous interval $(\tmins_n,\tmaxs_n)$ we may have either
vertical rearrangement (top figure; the segments over
$(\tmax_{n-1},\tmin_{n-1})$ and $(\tmin_n^*, \tmin_n)$ are both critical) or
no vertical rearrangement (bottom figure; only the segment over $(\tmax_n, \tmin_n)$ is critical).}
\label{fig:slopes}
\end{figure}
Our aim is to define inductively a polygonal curve that we shall subsequently 
prove (Theorem 
\ref{teo:the_polygonal_is_the_graph_of_ell})
to be the graph of the function $\ell$ with the initial condition 
\eqref{eq:inidat_Riemann}.

Inspired by the computations in Example \ref{exa:the_Riemann_pbm}
let us first define a real sequence $(\alpha_n)_{n\geq 0}$ and a sequence
 $(\beta_n)_{n\geq 0}$,
with $\beta_n \in [-\infty,0) \cup (0,+\infty)$, that will be
shown to be the slopes of the increasing (resp. decreasing\footnote{In Lemma 
\ref{teo:bounds_on_beta} we shall
prove that $\beta_n \in [-\infty,0)$.}) 
segments of the graph of $\ell$; it is convenient to
introduce also a real sequence
$(\beta_n^*)_{n \geq 1}$. Specifically:
\begin{Definition}
[\textbf{The sequences $(\alpha_n)$, $(\beta_n^*)$, $(\beta_n)$}]
We set
\begin{equation}\label{eq:slopes_initializing}
\alpha_0 := 1 , \qquad \beta_0 := -\infty ,
\end{equation}
and for any $n \geq 0$,
\begin{equation}\label{eq:slopes}
\begin{aligned}
&\frac{1}{\alpha_{n+1}} := \frac{1}{\alpha_n} + 1, \qquad \text{i.e. }
    \alpha_n = \frac{1}{n+1},
\\
& \frac{1}{\beta_{n+1}^*}
  :=\begin{cases}
  \frac{1}{\beta_n} + 1
  & {\rm if}~
  \beta_n \neq -1,0,
\\
1
  & {\rm if}~
  \beta_n  = -\infty, 
\\
  0^-
  & {\rm if}~
  \beta_n = -1,
  \end{cases}
\\
&
  \beta_{n+1} :=
  \begin{cases}
  \beta_{n+1}^* \qquad & 
{\rm if}~ \beta_{n+1}^* \leq 0,
  \\
  \alpha_{n+1} + \alpha_{n+2} - \beta_{n+1}^* 
& 
{\rm if}~  \beta_{n+1}^* > 0,
  \end{cases}
\end{aligned}
\end{equation}
where the value $0^-$ indicates that when computing its reciprocal we 
select
the value $-\infty$, i.e., $\beta_{n+1}^* = -\infty$.

\end{Definition}
For instance, $1/\beta_1^* = 1$, $\beta_1 = \alpha_1 + \alpha_2 - 1 = -1/6$.

Set 
\begin{equation}
\tmax_0 := 0, \qquad \lmax_0 := 0,
\qquad
\tmin_0 := 0, \qquad \lmin_0 := 1.
\label{eq:initial_conditions_t_M}
\end{equation}
Now, we iteratively define nine sequences
(the first five sequences could be considered as just
auxiliary), as follows: 
\begin{Definition}[\textbf{The sequences 
$(\tmins_n)$, 
$(\lmins_n)$,
$(\tmaxs_n)$, 
$(\lmaxs_n)$, 
$(\Delta_n^*)$, 
$(\tmax_n)$, 
$(\lmax_n)$, 
$(\tmin_n)$, 
$(\lmin_n)$}]
For $n \geq 0$ define the quantities:
\begin{equation}\label{eq:tmin_tmax}
\tmins_{n+1} := \tmin_n + \lmin_n , \quad
\lmins_{n+1} := \lmin_n , \quad
\tmaxs_{n+1} := \tmax_n + \lmax_n , \quad
\lmaxs_{n+1} := \lmax_n,
\end{equation}
$$
\Delta^*_n := \tmaxs_{n+1} - \tmins_{n+1}.
$$
Furthermore:
\begin{itemize}
\item[(i)]
if $\Delta^*_n \leq 0$ 
we set
\begin{equation}\label{eq:tmax_tmin_no_rearrangement}
\begin{aligned}
& \tmax_{n+1} := \tmaxs_{n+1} + \delta_n, \quad
\lmax_{n+1} := \lmaxs_{n+1} + 2\delta_n, 
\\
& \tmin_{n+1} := \tmins_{n+1} , \qquad \quad \ \  
\lmin_{n+1} := \lmins_{n+1},
\end{aligned}
\end{equation}
where $\delta_n$ is $1$ if $n=0$, zero otherwise;
\item[(ii)]
if $\Delta^*_n > 0$ we set
\begin{equation}\label{eq:tmax_tmin_rearrangement}
\begin{aligned}
\tmax_{n+1} := \tmins_{n+1} , \quad
& \lmax_{n+1} := \lmaxs_{n+1} - \Delta^*_n \alpha_n,
\\
\tmin_{n+1} := \tmaxs_{n+1} , \quad
& \lmin_{n+1} := \lmins_{n+1} + \Delta^*_n \alpha_{n+1} .
\end{aligned}
\end{equation}
\label{def:seqences_t_m}
\end{itemize}
\end{Definition}
For instance 
$\Delta_0^*=-1$, 
$\tmax_1 = 1$, 
$\lmax_1=2$,
$\tmin_1 = 1$, 
$\lmin_1=1$, 
$\Delta_1^*=1$, 
$\tmax_2 =2$, 
$\lmax_2=  3/2$,
$\tmin_2 = 3$, 
$\lmin_2=4/3$,
$\tmax_3 = 7/2$, 
 $\lmax_3=  3/2$. 

Clearly,
 for all $n \geq 0$
$$
0 < \tmax_n \leq \tmax_{n+1}, \quad  
0 < \tmin_{n} \leq \tmin_{n+1}, \quad
\lmax_{n+1} \leq \lmax_n \leq 2, \quad
\lmin_{n+1} \geq \lmin_n \geq 1.
$$

\begin{Remark}[\textbf{$\Delta_n^*$ as an indicator of rearrangement}]\rm
We shall prove in Proposition \ref{prop:different_form} that 
$\tmax_n \leq \tmin_n$ for all $n \geq 0$; this however does not
entail $\tmaxs_{n+1} \leq \tmins_{n+1}$.
If this is not the case we shall need to perform the so-called vertical rearrangement, as explained next\footnote{
For example, the first rearrangement happens for 
$n+1= 2$, leading to $M_{n+1}/m_{n+1} = (3/2)^2 \cdot 1/2 = 9/8 > 1$; 
the second one is for
$n+1 = 9$, leading to 
$M_9/m_9 = (10/9)^2 \cdot 8/9 = 800/729 > 1$, and the third one
for 
$n+1 = 11$, with 
$M_11/m_11 = (12/11)^2 \cdot 729/800 = 6561/6050  >  1$.
}:
so if $\Delta^*_n \leq 0$ 
there will be no vertical rearrangement,
whereas 
if $\Delta^*_n > 0$ there will be vertical rearrangement.

In Proposition 
\ref{prop:induction} 
we shall see that condition $\tmaxs_{n+2} > \tmins_{n+2}$ is equivalent to $\beta_{n+1}^* > 0$.
\end{Remark}

We now use the sequences defined in \eqref{eq:tmin_tmax}, \eqref{eq:tmax_tmin_no_rearrangement},
\eqref{eq:tmax_tmin_rearrangement} as coordinates of points in the $(t,y)$-plane,
setting
$$
\gmin_n := (\tmin_n,\lmin_n),
\quad
\gmax_n := (\tmax_n,\lmax_n),
\quad
\gmins_n := (\tmins_n,\lmins_n), 
\quad
\gmaxs_n := (\tmaxs_n,\lmaxs_n), 
$$
and observe that from 
\eqref{eq:tmin_tmax}
that 
\begin{equation}\label{eq:transformg}
\gmins_{n+1} = \affinetransformation(\gmin_n), \qquad
\gmaxs_{n+1} = \affinetransformation(\gmax_n),
\end{equation}
where $\affinetransformation$ is the affine transformation
of 
Section \ref{subsec:transformation_R_and_rearrangement}.

\begin{Definition}[\textbf{Polygonals $\mathcal {P}^*$ and $\mathcal P$}]
The polygonal $\mathcal{P}^*$ is  defined by the points
\begin{equation}\label{eq:polPstar}
\gmaxs_1, \gmins_1, \gmaxs_2, \gmins_2, \dots,
\gmaxs_{n+1}, \gmins_{n+1}, \dots 
\end{equation}
and the polygonal $\mathcal{P}$ is 
defined by the points
\begin{equation}\label{eq:polP}
\gmax_0, \gmin_0, \gmax_1, \gmin_1, \dots,
\gmax_n, \gmin_n, \dots
\end{equation}
\end{Definition}
We will show that the function having graph 
$\mathcal{P}$ 
is the result of the vertical rearrangement 
applied to the open set enclosed by $\mathcal{P}^*$ and the
positive $t$-axis, after addition of the initial datum $\inidat$. See 
also Figs. \ref{fig:slopes} and \ref{fig:riemann_specific_times}.

\begin{Proposition}[\textbf{Characterization}]\label{prop:different_form}
The sequences $(\tmax_n)$, $(\tmin_n)$, $(\lmax_n)$,
$(\lmin_n)$, $(\lmax_n^*)$, $(\lmin_n^*)$ 
defined in \eqref{eq:initial_conditions_t_M} and in Definition~\ref{def:seqences_t_m} can be
characterized as follows: for all $n >0$ 
\begin{equation}\label{eq:charat}
\tmax_{n+1}=\min\{\tmax_{n}+\lmax_{n},\tmin_{n}+\lmin_{n}\},
\qquad
\tmin_{n+1}=\max\{\tmax_{n}+\lmax_{n},\tmin_{n}+\lmin_{n}\},
\end{equation}
\begin{equation}\label{eq:charam}
\lmax_{n+1}
=\alpha_n(1+\tmax_{n+1}),
\qquad
\lmin_{n+1}
=\alpha_{n+1}(1+\tmin_{n+1}),
\end{equation}
\begin{equation}\label{eq:charamstar}
\lmax_{n+1}^*
=\alpha_n(1+\tmax_{n+1}^*),
\qquad
\lmin_{n+1}^*
=\alpha_{n+1}(1+\tmin_{n+1}^*).
\end{equation}
In particular, $\lmin_{n+1}>0$,  $\lmax_{n+1} >0$, and 
$$
\tmax_{n+1} \leq \tmin_{n+1}.
$$
\end{Proposition}

\begin{proof}
Let $n>0$. If $\Delta^*_n\le 0$, i.e., $\tmax_n+\lmax_n\le
\tmin_n+\lmin_n$, 
then 
$\tmax_{n+1}=\tmax_n+\lmax_n$
and $\tmin_{n+1}=\tmin_n+\lmin_n$, while
if $\Delta^*_n> 0$, i.e.,
 $\tmax_n+\lmax_n>
\tmin_n+\lmin_n$, then 
$\tmax_{n+1}=\tmin_n+\lmin_n$ and
$\tmin_{n+1}=\tmax_n+\lmax_n$, and \eqref{eq:charat}
follows.

For $n \geq 0$ let $y = r_n(t) := \alpha_n (1+t)$ denote the equation of 
the line with slope $\alpha_n$ through $(-1,0)$. With a small abuse of 
notation we also denote by $r_n$ the line itself. The key observation is 
that the affine transformation $\affinetransformation$ maps line $r_n$ 
onto line $r_{n+1}$.
Given that $\gmin_0, \gmax_1 \in r_0$, we now argue by induction to prove that
for any $n \geq 1$ points 
$\gmin_n, \gmins_n, \gmax_{n+1}, \gmaxs_{n+1} \in r_n$.
Now suppose $\gmin_n, \gmax_{n+1} \in r_n$, then \eqref{eq:transformg} implies
that $\gmins_{n+1}, \gmaxs_{n+2} \in r_{n+1}$.
If $\gmin_{n+1} \neq \gmins_{n+1}$, i.e. 
rearrangement occurs at step $n$ (resp. $\gmax_{n+2} \neq \gmaxs_{n+2}$, i.e. 
rearrangement occurs at step $n+1$), equation 
\eqref{eq:tmax_tmin_rearrangement} tells us that the segment 
$\gmins_{n+1}\gmin_{n+1}$ (resp. $\gmax_{n+2}\gmaxs_{n+2}$) has slope 
$\alpha_n$, so that the ``unstarred'' points are the result of suitably moving
along line $r_{n+1}$ the corresponding ``starred'' point.
This concludes the induction step and the claim follows together with the
equivalent assertions \eqref{eq:charam} and \eqref{eq:charamstar}.
\end{proof}

\begin{Remark}\label{rem:separate}\rm
It follows from Proposition~\ref{prop:different_form},
using e.g. $\frac{\alpha_n}{\alpha_{n+1}} = 1 + \alpha_n$
that the sequences $(\tmin_n)$ and $(\tmax_n)$ can be decoupled
from the sequences $(\lmin_n)$ and $(\lmax_n)$, and for all $n >0$
\begin{equation}\label{eq:decoupling}
\begin{aligned}
&\tmax_{n+1}=\min\Bigl\{\alpha_{n-1}+\frac{\alpha_{n-1}}{\alpha_n} \tmax_{n}, 
\alpha_n
+\frac{\alpha_n}{\alpha_{n+1}} \tmin_{n}\Bigr\},
\\
&\tmin_{n+1}=\max\Bigl\{\alpha_{n-1}+\frac{\alpha_{n-1}}{\alpha_n} 
\tmax_{n}, \alpha_n
+\frac{\alpha_n}{\alpha_{n+1}} \tmin_{n}\Bigr\}.
\end{aligned}
\end{equation}
Concerning the values $(\lmin_n)$, $(\lmax_n)$, in the interesting case $\Delta_n^* > 0$
(rearrangement), enforcing \eqref{eq:tmin_tmax}, \eqref{eq:tmax_tmin_rearrangement}, \eqref{eq:charam}, \eqref{eq:charamstar},
we obtain
\begin{equation}\label{eq:in_the_interesting_case}
\Delta_n^* > 0 \implies \quad
\lmax_{n+1} = \frac{\alpha_n}{\alpha_{n+1}}
\lmin_n , \quad
\lmin_{n+1} = \frac{\alpha_{n+1}}{\alpha_n} \lmax_n
\end{equation}
leading to the remarkable fact that 
$\lmin_{n+1}\lmax_{n+1} = \lmin_n\lmax_n$
regardless whether we have rearrangement.
Consequently 
\begin{equation}\label{eq:prod_2}
\lmin_n \lmax_n = \lmin_1 \lmax_1 = 2 \qquad \forall n \geq 1.
\end{equation}
Recalling that $(\lmin_n)$ is non-decreasing and $(\lmax_n)$ is non-increasing, using \eqref{eq:charat} we
get
$$
\lmin_n \geq 1, \qquad \lmax_n \leq 2,
$$
$$
\tmin_n \geq \tmin_1 + \lmin_1 + \lmin_2 + \dots + \lmin_{n-1} \geq \tmin_1 + (n-1)\lmin_1 = n ,
$$
$$
\tmax_n \leq \tmax_1 + \lmax_1 + \lmax_2 + \dots + \lmax_{n-1} \leq \tmax_1 + (n-1)\lmax_1 = 2n - 1.
$$
\end{Remark}

\begin{Lemma}[\textbf{Sign of $\beta_n$}]\label{teo:bounds_on_beta}
We have $\beta_n \in [-\infty,0)$ and 
$\beta_n \in [-\infty, -1 + \alpha_n + \alpha_{n+1}]$
for all $n \geq 0$.
\end{Lemma}

\begin{proof}
The result is clearly true for $n=0$.
We first prove $\beta_n \in [-\infty,0)$ for all $n \in \NN$.
Suppose by induction that $\beta_n \in [-\infty,0)$ for some $n \geq 0$.
If $\beta_n \in [-1,0)$ then \eqref{eq:slopes} gives
$\beta_{n+1} = \beta_{n+1}^* = \frac{\beta_n}{\beta_n + 1} < 0$.
Suppose now that $\beta_n < -1$, i.e. $\beta_{n+1}^* \geq 1$.
We have that $\alpha_n + \alpha_{n+1}$ is (strictly) decreasing and
$\alpha_1 + \alpha_2 = \frac{5}{6}$, so that \eqref{eq:slopes}, second case in definition of $\beta_{n+1}$,
entails $\beta_{n+1} = \alpha_{n+1} + \alpha_{n+2} - \beta_{n+1}^* \leq
\frac{5}{6} - 1 < 0$.

Now, let us prove
$\beta_n \in [-\infty, -1 + \alpha_n + \alpha_{n+1}]$ for all $n \in 
\NN$.
We first observe that in case of no vertical
rearrangement ($\beta_n^* < 0$)
we have $\beta_n < \beta_{n-1}$, so that we need to prove the result 
only in case of vertical rearrangement
($\beta_n < 0 < \beta_n^*$).
Definition \eqref{eq:slopes} of $\beta_n^*$ entails 
$1/\beta_n^* = 1/\beta_{n-1} + 1 \leq 1$.
It follows $\beta_n^* \geq 1$, so that $\beta_n = -\beta_n^* + \alpha_n + \alpha_{n+1}
\leq -1 + \alpha_n + \alpha_{n+1}$ (which is $<0$ if $n \geq 1$). 
\end{proof}

In connection with the next result,
it is useful to notice that a segment/line of slope $\gamma$ 
is transformed via $\affinetransformation$ into a segment/line
of slope $\gamma'$ with $\frac{1}{\gamma'} = \frac{1}{\gamma} + 1$ (with the position $\frac{1}{0} = \infty$).

\begin{Proposition}[\textbf{Slopes of polygonals}]\label{prop:induction}
For any $n \geq 0$ we have:
\begin{itemize}
\item[(i)]
points $\gmin_n$ and $\gmax_{n+1}$ lie of a line having slope $\alpha_n$;
\item[(ii)]  the segment with 
endpoints $\gmaxs_{n+1}$, $\gmins_{n+1}$ has slope $\beta^*_n$;
\item[(iii)]  the segment with endpoints $\gmax_{n+1}$, $\gmin_{n+1}$ has 
slope $\beta_n$;
\item[(iv)]  $\Delta^*_n > 0 \iff \beta^*_n > 0$;
\item[(v)] 
$\lmin_{n+1} < \lmax_{n+1}$;
\item[(vi)]  $\tmin_n < \tmax_{n+1}$.
\end{itemize}
\end{Proposition}
\begin{proof}
To prove (i)-(iv) we argue by induction.
All assertions (i) to (iv) are true for $n=0$.
Suppose that all claims are true for $n$ and let us prove them for $n+1$.
%
%
%
Segment $s$ from $\gmins_n$ to $\gmaxs_{n+1}$ is 
obtained from segment from $\gmin_{n-1}$ to $\gmax_n$
via the affine transformation $\affinetransformation$, hence it has inverse slope given by
$\frac{1}{\alpha_{n-1}} + 1$, so that its slope is $\alpha_n$.
Point $\gmax_{n+1}$ (resp. $\gmin_n$) either coincides with $\gmaxs_{n+1}$ (resp. $\gmins_n$), the right (resp. left)
endpoint of $s$, or,
in case of vertical rearrangement, can be checked to lie on segment $s$ using the first equation in
\eqref{eq:tmax_tmin_rearrangement} (resp. the second equation taking $n$ in place of $n+1$).
In either case the corresponding segment from $\gmin_n$ to $\gmax_{n+1}$
has itself slope $\alpha_n$.
Likewise segment $r$ with endpoints $\gmaxs_{n+1}$, $\gmins_{n+1}$ has
inverse slope $\frac{1}{\beta_{n-1}} + 1$, i.e. slope $\beta_n^*$.
This entails $\Delta^*_n > 0 \iff \beta^*_n > 0$.
In case of no vertical rearrangement
($\tmaxs_{n+1} \leq \tmins_{n+1}$) we have $\beta_n = \beta_n^*$ and we have finished.
In case of vertical rearrangement we have both $\Delta^*_n > 0$ and $\beta^*_n > 0$,
the slope of $\gmax_{n+1} \gmin_{n+1}$ can be obtained by computing the vertical displacement
of a point that starts from $\gmax_{n+1}$, moves to $\gmaxs_{n+1}$ along a segment with slope $\alpha_n$,
then to $\gmins_{n+1}$ (slope $\beta^*_n$), then to $\gmin_{n+1}$ (slope $\alpha_{n+1}$).
The horizontal displacements have size $\Delta^*_n$, $-\Delta^*_n$, $\Delta^*_n$ for a combined slope of
$\alpha_n + \alpha_{n+1} - \beta^*_n$, consistently with the definition of $\beta_n$
(Figure \ref{fig:slopes}). This concludes the proof of (i)-(iv).

\noindent%
We are left with (v) and (vi), which we also prove by induction,
observing that the claims are true for $n=0$.

(v). If $\Delta^*_n \leq 0$ we have 
$\lmin_{n+1} 
=\lmin_{n} 
< \lmax_{n}=\lmax_{n+1}$. 
If $\Delta^*_n > 0$ we use that $\beta_n <0$ (Lemma 
\ref{teo:bounds_on_beta}) and item (iii): 
we have $\lmin_{n+1}-\lmax_{n+1}= \beta_n \Delta_n^* <0$.  

(vi).
Since $\lmin_{n+1} \geq \lmin_n$ we get $\lmax_{n+1} > \lmin_n$ and from 
item (i) and the fact that $\alpha_n > 0$
we directly obtain 
\begin{equation}\label{eq:directly_obtain}
\tmax_{n+1} - \tmin_n = \frac{1}{\alpha_n} (\lmax_{n+1} - \lmin_n) > 0.
\end{equation}
\end{proof}

The sequence of points so constructed defines the polygonal $\mathcal{P}$ 
 which
is the graph of a function $\ell$.
We finally prove:

\begin{Theorem}\label{teo:the_polygonal_is_the_graph_of_ell}
The function $\ell$ so defined verifies \eqref{eq:how_to_reconstruct_l_from_itself}.
\end{Theorem}

\begin{proof}
In view of Section \ref{subsec:transformation_R_and_rearrangement},
we only need to compute the vertical rearrangement of the multivalued function having
polygonal $\mathcal{P}^*$ as graph.
We have vertical rearrangement only in segments $(\tmax_{n+1},\tmin_{n+1})$ where $\Delta^*_n > 0$.
Since each of the three values varies linearly, the result can be computed by linear interpolation
between the values $\lmax_{n+1}$ and $\lmin_{n+1}$, giving, except for the interval $[0,1]$, the
polygonal $\mathcal{P}$ as graph.
The result follows after the addition of $\inidat$.
\end{proof}

In the end, 
$(\tmin_n)_{n \geq 1}$ is
the sequence of positive local minima of $\uM$,
$\lmin_n:= \uM(\tmin_n)$,
$\tmax_{n+1} \in (\tmin_n,\tmin_{n+1}]$ is 
the local maximum of $\uM$ in 
$[\tmin_n,\tmin_{n+1}]$, and 
$\lmax_{n+1}:=\uM(\tmax_{n+1})$.
Equality $\tmax_{n+1} =
\tmin_{n+1}$ holds only at a jump,
where $\uM(\tmin_{n+1})$ (resp. $\uM(\tmax_{n+1})$) indicates here the
smallest (resp. largest) value of $\uM$. 

\begin{Remark}[\textbf{Recovering geometrically the solution
from $\uM$}]\label{rem:from_l_to_v}\rm
Having constructed, although not in closed form, function $\ell$,
 we  are in a position to recover from it
the solution $\pressol (\tau, \cdot)$ at any given positive
time $\tau$.
To this aim we can resort to Remark \ref{rem:alternative_expression}, 
which provides a convenient way to achieve
our goal.
Particularly important is the term $|\{s\in[\tau,\xi] : \uM(s) \leq \xi - s \}|$ in formula \eqref{eq:part_imp}, which can actually be written
in a slightly different way,
 if we consider the image 
$$
A^*:= \affinetransformation({\rm subgraph}_+(\uM)).
$$ 
$A^*$ is bounded by the polygonal $\mathcal{P}^*$ 
and the positive $\tau$-axis, 
it gets intersected with the vertical
line lifted from abscissa $\xi$, and more specifically
$$
|\{s\in[\tau,\xi] : \uM(s) \leq \xi - s \}| = |\{(\xi,s): s \in [0,\xi - \tau]\} \setminus A^*|
$$
which, in view of \eqref{eq:how_to_reconstruct_l_from_itself} and \eqref{eq:obv}, can be recognized as a ``partial''
vertical rearrangement of $A^*$ happening only inside $\soct$.

This is illustrated in Figure \ref{fig:riemann_specific_times}, containing 
the evolution of
the present Example at specific times.
At time $\tau = 0.8$ the ``zig-zag'' in polygonal $\mathcal{P}^*$, in particular the triangular ``void'' that
is bound to be filled by the vertical rearrangement, is completely below the dashed line with slope $1$.
This implies that no rearrangement is taking place yet in the reconstruction of $\pressol (\tau,\cdot)$
and we simply use the top segment in the zig-zag of $\mathcal{P^*}$.
This gives the same result as re-adding the size of the void triangle intersected with the vertical line at $\xi$
to the quantity $\ell(\xi) - (\xi - \tau)$.
At time $\tau = 1.4$ the zig-zag (the ``void'' triangle) is partially above the dashed line with slope $1$.
This implies partial vertical rearrangement of only the part of $A^*$ lying above this line, which is
the same as re-adding to $\ell(\xi) - (\xi - \tau)$ that part of the vertical segment at $\xi$ that
is in the void triangle and below the line at $45$ degrees.
At time $\tau = 1.8$ the zig-zag is completely above the line at $45$ degrees, so that there is no portion of
the void triangle below that line and there is no contribution from $\{(\xi,s): s \in [0,\xi - \tau]\} \setminus A^*$
in the reconstruction of $\pressol$.
Finally, time $\tau=3.2$ illustrates the situation where there is no vertical rearrangement at all in function $\uM$.
Clearly in this case we also have no contribution from $\{(\xi,s): s \in [0,\xi - \tau]\} \setminus A^*$.

It should be noted that the Lyapunov functional decreases strictly exactly when we have only partial
rearrangement (e.g. in a neighborhood of time $\tau = 1.4$ in Figure \ref{fig:riemann_specific_times}).
\end{Remark}

\begin{figure}
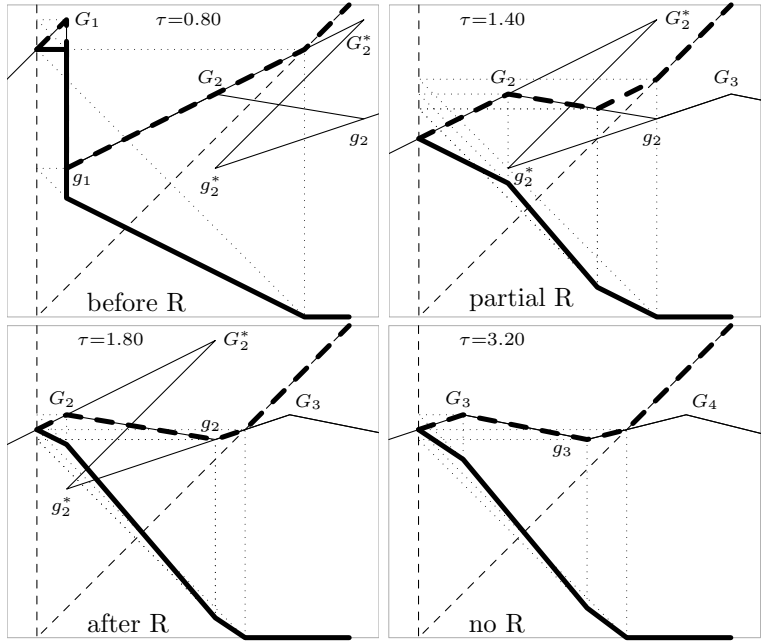

\begin{center}
\includegraphics{\figdir/riemannx_times-6.mps}
\includegraphics{\figdir/riemannx_times-9.mps}
\includegraphics{\figdir/riemannx_times-11.mps}
\includegraphics{\figdir/riemannx_times-18.mps}
\end{center}
\caption{\small Various typical situations in the 
reconstruction of $\pressol(\tau,\cdot)$ at specific times during
the evolution of the Riemann problem. The bold polygonal represents the graph 
of $\pressol(\tau,\cdot)$,
the thin lines display part of the polygonals $\mathcal{P}$ and $\mathcal{P}^*$ (defined by \eqref{eq:polP} and \eqref{eq:polPstar} respectively), whereas
the bold dashed line is the graph of $\pressol(\tau,\xi) + (\xi - \tau)$. 
R stands for (vertical)
rearrangement.}
\label{fig:riemann_specific_times}
\end{figure}

%
\makeatletter
\begin{figure}
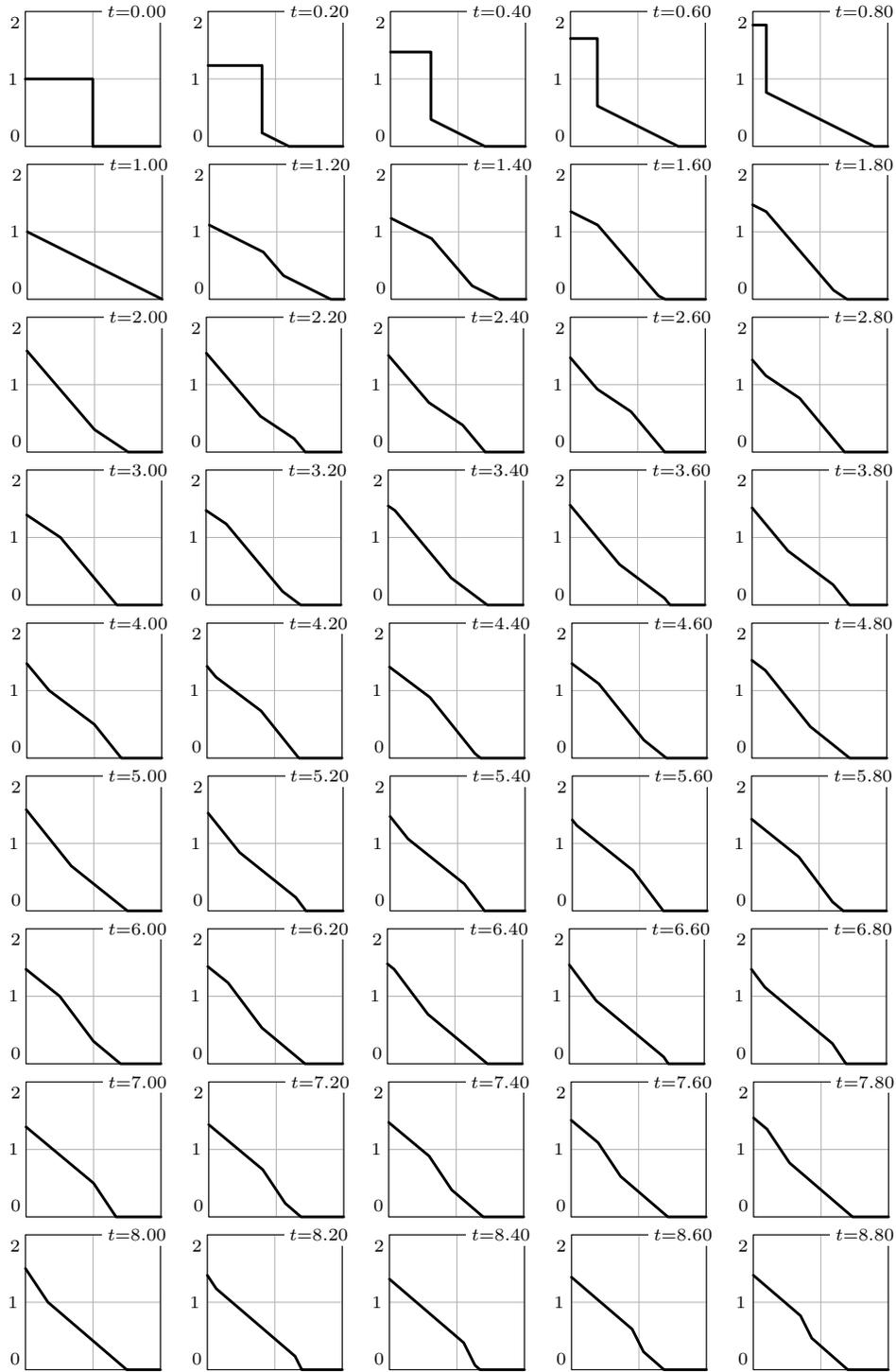

\begin{center}
\newcounter{itime}\stepcounter{itime}
\@whilenum\value{itime} < 45\do
{
\noindent%
\hfill\stepcounter{itime}\includegraphics{\figdir/riemann_times-\theitime.mps}
\hfill\stepcounter{itime}\includegraphics{\figdir/riemann_times-\theitime.mps}
\hfill\stepcounter{itime}\includegraphics{\figdir/riemann_times-\theitime.mps}
\hfill\stepcounter{itime}\includegraphics{\figdir/riemann_times-\theitime.mps}
\hfill\stepcounter{itime}\includegraphics{\figdir/riemann_times-\theitime.mps}\hfill
\\
}
\end{center}
\caption{\small 
Time-movie of the solution of the Riemann problem ($\inidat = 1_{[0,1)})$,
up to time $t = 8.80$. 
}
\label{fig:Riemann_movie_I}
\end{figure}
\begin{figure}
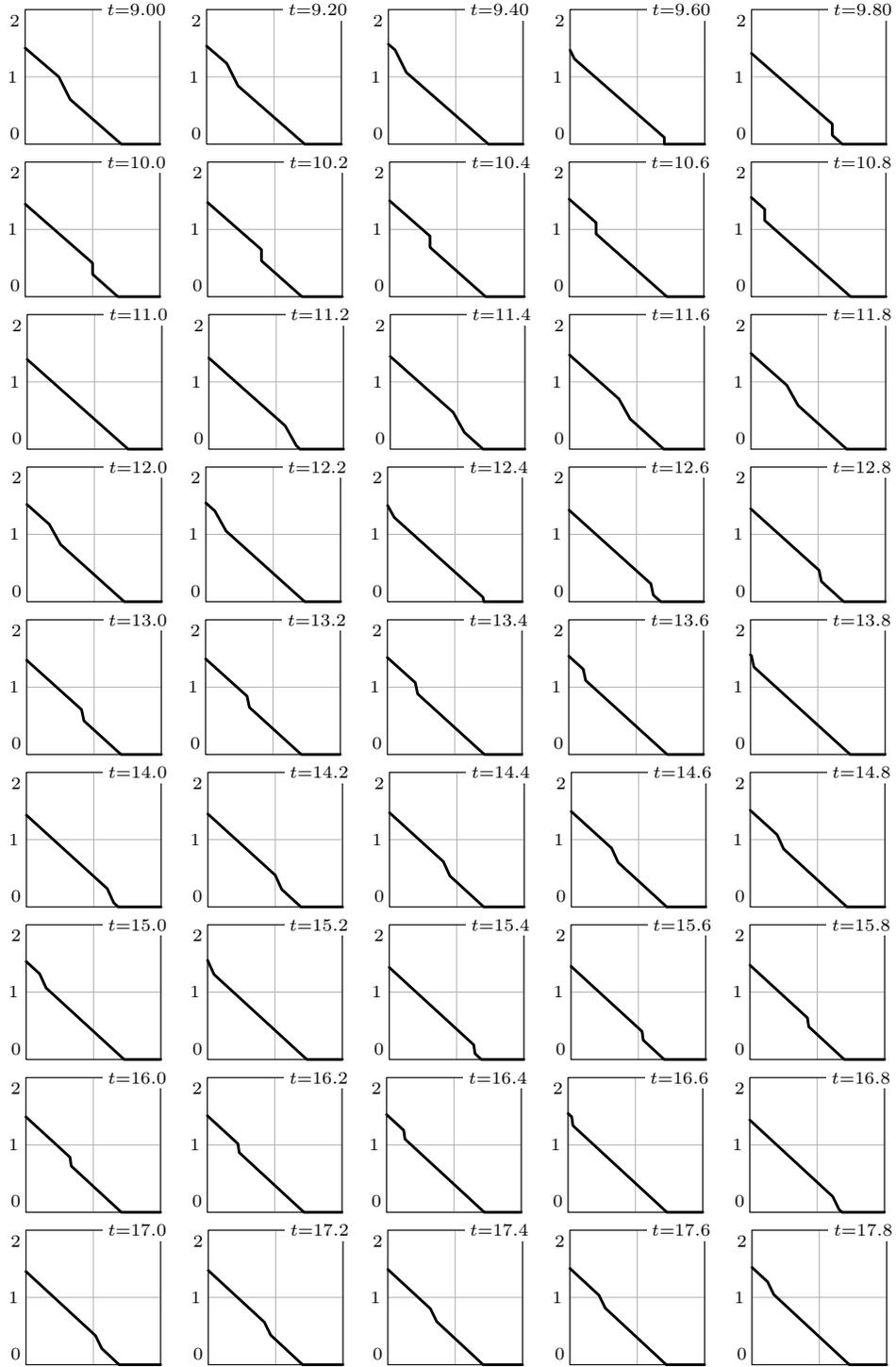

\begin{center}
\@whilenum\value{itime} < 91\do
{
\noindent%
\hfill\stepcounter{itime}\includegraphics{\figdir/riemann_times-\theitime.mps}
\hfill\stepcounter{itime}\includegraphics{\figdir/riemann_times-\theitime.mps}
\hfill\stepcounter{itime}\includegraphics{\figdir/riemann_times-\theitime.mps}
\hfill\stepcounter{itime}\includegraphics{\figdir/riemann_times-\theitime.mps}
\hfill\stepcounter{itime}\includegraphics{\figdir/riemann_times-\theitime.mps}\hfill
\\
}
\end{center}
\caption{\small Time-movie of the solution of the Riemann problem ($\inidat = 1_{[0,1)})$,
from time $t = 9.00$ up to time $t=17.8$. }
\label{fig:Riemann_movie_II}
\end{figure}
\makeatother

Once $\uM$ is known, the solution 
$u$ is uniquely determined via Theorem \ref{teo:reconstructing_u_hat_from_l}
(and using the map $\Phi^{-1}$); for instance
(see Figs. \ref{fig:Riemann_movie_I} and \ref{fig:Riemann_movie_II}):

\begin{equation}\label{eq:before_than_one}
u(t,x) = 
\begin{cases}
1+t & t \in [0,1), ~ 0 \leq x < 1-t,
\\
-\frac{1}{2} x + \frac{1}{2} t + \frac{1}{2}
& t \in [0,1), ~ 1-t < x \leq 1+t,
\\
0 
& t \in [0,1), ~x \geq 1+t.
\end{cases}
\end{equation}
Note that $u$ is discontinuous along $\{(t,x): t \in [0,1], x+t=1\}$.
Also
$$
\lim_{s \uparrow 1} u(s, x) 
= \Big(1 -\frac{1}{2}x\Big) \vee 0, \qquad x \in (0,+\infty),
$$
which is piecewise affine and Lipschitz, and
\begin{equation}\label{eq:before_than_two}
\begin{aligned}
& u(t,x) = 
\begin{cases}
\frac{t}{2} -\frac{1}{2}(x-1)&  0 \leq x \leq 2-t,\\
\frac{11}{6}-\frac{7}{6}x-\frac{t}{6}
&  2-t < x \leq \frac{t}{2}+\frac{1}{2},\\
\frac{3}{2}-\frac{x}{2} -\frac{t}{2}
&  \frac{t}{2}+\frac{1}{2} < x \le 3-t,\\
0 
& x > 3-t,
\end{cases} \qquad \qquad t \in [1,5/3].
\end{aligned}
\end{equation}

Recalling \eqref{eq:derivative_Lyap}, one checks that 
$\tau \in [0, +\infty) \to \Lyap(\wu(\tau))$ is Lipschitz and nonincreasing,
 strictly decreasing (quadratically) only on those intervals
where $L_\uM$ becomes strictly
decreasing. In particular, since
$\Lyap(\inidat) = 
 1$,
$\Lyap(u(1^-, \cdot)) = 1$,
it follows
$\Lyap(u(t, \cdot)) = 1$ for all $t \in [0,1]$.
In addition
$\Lyap\left(u\Bigl(\frac{5}{3},\cdot\Bigr)\right)
=\frac{17}{18}<1$.
For $t \in (5/3, 11)$ the value $\Lyap(u(t,\cdot)$
remains constant, and
next it decreases slightly in the 
interval $[11,11+3/20]$.

\begin{figure}
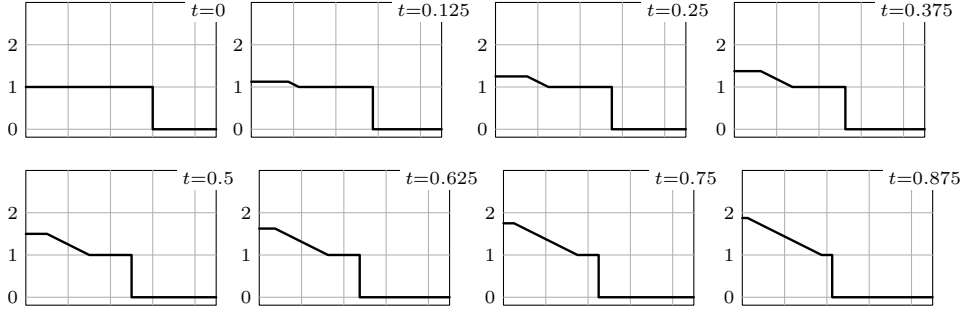

\hbox to\hsize{\(22)\hfil\(23)\hfil\(24)\hfil\(25)}
\bigskip
\hbox to\hsize{\(26)\hfil\(27)\hfil\(28)\hfil\(29)}
\caption{{\small Time-slices of the solution $u_1(t,x)$
described in Example~\ref{exa:Riemann_problem_with_a_parameter} ($a=3$),
for $t$ in the allowed range: the slope of the oblique segment
is $-1/2$.}}
\end{figure}

\nada{
\begin{Remark}[\textbf{Conservation of mass}]\label{rem:in_principle}\rm
Conservation of mass of the solution $\wu$ can be checked
from the graph of $L_\uM$ in 
Fig. \ref{fig:riemann-ell}: for instance, at time $t=2$ the 
mass of $\wu$ is the area of the triangle
with vertices $(1/2,2), (1,3), (1,2)$ plus the area of the 
triangle with vertices $(1,2), (2,7/2), (2,2)$.
\end{Remark}
}

\subsection{Asymptotic properties of the solution}
A direct consequence of 
the results of the previous section
is 
 the convergence of
the solution to a stationary configuration
(described in Example~\ref{exa:stationary_solutions})
in infinite time.

From (v) of Proposition \ref{prop:induction} it follows
\begin{equation}\label{eq:lambda1_lambda2}
1 < \lambda^- := 
\lim_{n\to+\infty} \lmin_n \leq 
\lambda^+:=\lim_{n\to+\infty} \lmax_n < 2.
\end{equation}

This immediately leads to the desired convergence property as $t \to +\infty$.
\begin{Proposition}[\textbf{Asymptotic convergence to the stationary solution}]
We have 
\begin{equation}\label{eq:sqrt2}
\lambda^+ = \lambda^-=\sqrt 2
\end{equation}
and $\lim_{t \to +\infty} \Vert u(t,\cdot) - 
u_{{\rm stat}}(\cdot)\Vert_{L^\infty([0,+\infty))}=0$,
where
$u_{{\rm stat}}(x) = \max\{\sqrt{2}-x,0\}$ is the stationary
solution with unit mass.
\end{Proposition}
\begin{proof}
From the definition we have $\tmax_{n+1} \leq \tmaxs_{n+1}$.
Using $\tmax_n \leq \tmin_n$ (Proposition \ref{prop:different_form}) we obtain
the estimate
$$
\tmax_{n+1} - \tmin_n \leq \tmaxs_{n+1} - \tmax_n = \lmax_n \leq 2 .
$$
Then \eqref{eq:sqrt2} follows from
$$
\lim_{n \to \infty} (\lmax_{n+1} - \lmin_n) = \lim_{n \to \infty} \alpha_n (\tmax_{n+1} - \tmin_n) \leq 2 \lim_{n \to \infty} \alpha_n = 0,
$$
where we use \eqref{eq:directly_obtain}.
\nada{
Formula \eqref{eq:sqrt2} follows 
from \eqref{eq:lambda1_lambda2} and \eqref{eq:prod_2}.

Observe that, if $\Delta_n^*>0$, from \eqref{eq:in_the_interesting_case}
$$
0 < \lmax_{n+1} - \lmin_n=\frac{\alpha_{n+1}}{\alpha_n} \lmin_n - \lmin_n
= \alpha_n \lmin_n \leq 2\alpha_n.
$$

Then
$0=\lim_{n\to\infty}(\lmax_n-\lmin_n)=\lambda^+-\lambda^-$.
Now using \eqref{eq:prod_2}
we immediately get \eqref{eq:sqrt2}.
}

Hence $\lim_{t \to +\infty} \sup_{s >t} \vert \uM(s) - \sqrt{2}\vert=0$,
which,
 from \eqref{eq:recovering_of_u_hat},
 implies the desired convergence of the solution 
to $u_{\rm stat}$.
\end{proof}

\subsection{Final examples}\label{sec:final_examples}
\begin{Example}[\textbf{Parent of $1_{[0,1)}$: 
creation of a decreasing jump}]
Let 
$$
\inidat(x) = (2-2x) \vee 0 \qquad \forall x \geq 0.
$$
We have %
$\uM(t) = -t+2$ for $t \in (0,1]$, 
and $\uM(t) = t$ for $t \in [1,2]$. 
Applying transformation $\affinetransformation$ 
for  $t \in [0,1]$ 
produces the vertical segment $\{1\}\times 
[1,2]$.
We have 
$$
u(1,\cdot) = 
1_{[0,1)}.
$$
%
The flow for times larger than $1$ is then
the same as the flow of the Riemann problem.
\end{Example}

\begin{Example}[\textbf{Riemann problem with a parameter}]
\label{exa:Riemann_problem_with_a_parameter}
\rm
Let $\inidat = 1^{}_{[0,a)}$, with $a > 1$.
Let $T$ be obtained as the
intersection between the lines
$\{x- t=1\}$ and $\{x+  t=   a\}$, i.e.,
$T = \frac{a-1}{2}$, and set
$T_1 = \min(1,\frac{a-1}{2})$.
The solution is given, for any $(t,x) \in
[0,T_1] \times (0,+\infty)$ by
$$
u(t,x) = \begin{cases}
1+t & {\rm if}~
x \leq 1- t,
\\
-\frac{x}{2} + \frac{t}{2} + \frac{3}{2} & {\rm if}~
1-  t \leq x \leq 1+  t,
\\
1 & {\rm if}~
1+ t \leq x \leq a- t,
\\
0 & {\rm if}~
x + t\geq a.
\end{cases}
$$
%
\end{Example}
Stationary solutions can be reached also in finite time, as shown
in this final example; notice that here the positivity condition
\eqref{eq:inidat_larger_than_alpha} on $\inidat$ is not satisfied.

\begin{Example}[\textbf{Reaching a stationary solution in finite time}]
Let $\inidat$ be the linear interpolation of $\inidat(0)=0$,
$\inidat(1)= 1/2$ and $\inidat(2)=0$. Then 
$$
\uM(t) = \begin{cases}
t & {\rm if}~t \in [0,1],
\\
1 & {\rm if}~ t > 1,
\end{cases}
$$
and 
the corresponding solution 
$u$ reaches the stationary solution $\max\{1-x,0\}$ at time 
$t=1$.
\end{Example}


%

\end{document}